\documentclass[11pt]{amsart}

\usepackage{amsmath, amsthm, amsfonts, amssymb,accents, a4, enumerate}
\usepackage{srcltx,color,a4wide}

\theoremstyle{plain}
\newtheorem{theorem}{Theorem}[section]
\newtheorem{lemma}[theorem]{Lemma}
\newtheorem{corollary}[theorem]{Corollary}
\theoremstyle{remark}
\newtheorem{remark}[theorem]{Remark}

\numberwithin{equation}{section}

\def\a{\alpha}
\def\b{\beta}
\def\d{\delta}
\def\e{\varepsilon}
\def\eps{\epsilon}
\def\g{\gamma}
\def\G{\Gamma}
\def\l{\lambda}
\def\p{\partial}
\def\D{\Delta}

\def\E{\mbox{\rm e}}

\def\Odr{\mathcal{O}}

\def\di{\,\mathrm{d}}

\def\Ups{\Upsilon}

\def\I{\mathrm{I}}

\def\la{\langle}
\def\ra{\rangle}

     \newcommand{\CC}{\mathbb{C}}
     \newcommand{\EE}{\mathbb{E}}
     
     \newcommand{\NN}{\mathbb{N}}
     \newcommand{\PP}{\mathbb{P}}
     
     \newcommand{\RR}{\mathbb{R}}
     \newcommand{\ZZ}{\mathbb{Z}}

     \newcommand{\cH}{\mathcal{H}}

     \newcommand{\cQ}{\mathcal{Q}}
     
     \newcommand{\cT}{\mathcal{T}}

\newcommand{\BIGOP}[1]{\mathop{\mathchoice%
{\raise-0.22em\hbox{\huge $#1$}}%
{\raise-0.05em\hbox{\Large $#1$}}{\hbox{\large $#1$}}{#1}}}
\newcommand{\BIGboxplus}{\mathop{\mathchoice%
{\raise-0.35em\hbox{\huge $\boxplus$}}%
{\raise-0.15em\hbox{\Large $\boxplus$}}{\hbox{\large $\boxplus$}}{\boxplus}}}

\DeclareMathOperator{\dist}{dist}
\DeclareMathOperator{\supp}{supp}

\DeclareMathOperator{\Det}{det}
\newcommand{\per}{\mathrm{per}}

\newcommand{\hm}[1]{\leavevmode{\marginpar{\tiny%
$\hbox to 0mm{\hspace*{-0.5mm}$\leftarrow$\hss}%
\vcenter{\vrule depth 0.1mm height 0.1mm width \the\marginparwidth}%
\hbox to 0mm{\hss$\rightarrow$\hspace*{-0.5mm}}$\\\relax\raggedright #1}}}

 \thanks{ \copyright 2013 by the authors. Faithful reproduction of this article, is permitted for non-commercial purposes.
 {\today, \jobname.tex}}

\begin{document}

\allowdisplaybreaks

\title[Spectral properties of randomly curved  quantum waveguides]
{Low lying eigenvalues of randomly curved  quantum waveguides}
\author[Borisov and Veseli\'c]{Denis Borisov$^{1,2,3,4}$ and  Ivan Veseli\'c$^{1,5}$}

\begin{abstract}
We consider the negative Dirichlet Laplacian on an infinite waveguide embedded in $\RR^2$,
and finite segments thereof.  The waveguide is a perturbation of a periodic strip
in terms of a sequence of independent identically distributed random variables which
influence the curvature.
We derive explicit
lower bounds on the first eigenvalue of finite segments of the randomly curved waveguide
in the small coupling (i.e.~weak disorder) regime. This allows us to estimate the probability
of low lying eigenvalues, a tool which is relevant in the context of Anderson localization
for random Schr\"odinger operators.
\end{abstract}

\maketitle

\begin{itemize}
\item[(1)] \emph{Faculty of Mathematics, Technische Universit\"at Chemnitz, 09107 Chemnitz, Germany}
\item[(2)] \emph{Institute of Mathematics Of Ufa Scientific Center
of RAS, Chernyshevskogo, 112, Ufa, 450008, Russia}
\item[(3)] \emph{Department of Physics and Mathematics, Bashkir
State Pedagogical University, October rev. st.~3a, Ufa, 450000,
Russia}
\item[(4)] \emph{E-mail:}\ \texttt{borisovdi@yandex.ru}, \emph{URL:} \texttt{http://borisovdi.narod.ru/}
\item[(5)] \emph{URL:}\ \texttt{http://www.tu-chemnitz.de/mathematik/stochastik/}
\end{itemize}

\tableofcontents

\section{Introduction}
\label{s:introduction}

In this paper we study eigenvalues
of finite segments of randomly curved (quantum) waveguides.
More precisely we derive lower bounds on the bottom of the spectrum
of a waveguide segment of length $L$ with  perturbation of (a small) size $\kappa$.
The bounds are formulated in terms of the parameters $L$ and $\kappa$.
In this context lower bounds are more challenging than upper ones.
The waveguide consists of the set of points in $\RR^2$ which have a distance
smaller than $\pi/2$ (half the width of the waveguide)
to a random perturbation of a periodic curve in the plane.
We consider the negative Laplace operator on this set with appropriate boundary conditions.
As it has compact resolvent it has purely discrete, lower semi-bounded spectrum.

Since the perturbation in each periodicity cell of the waveguide is
determined by an individual scalar random variable,
the above bound can be considered as a multi-parameter variational problem.
The main difficulty arises from the fact that the length of the waveguide
is finite, but not fixed; indeed it tends to infinity.
This also means that the number of variational parameters becomes
arbitrarily large. To compensate this, we have to chose the size $\kappa$ of
the perturbations as a function of the length $L$. This is a special case of what is
commonly called a \emph{weak disorder} regime.

The estimates on lowest eigenvalues we derive are motivated by
the question whether a random Hamiltonian on a non-compact space has localized or
non-localized states. One specific question in this context is:
Is there a interval containing the bottom of the spectrum of the
random Hamiltonian where the eigenvalues lie dense and where there
is no continuous spectral component almost surely.
The two first models where this question has been answered positively
is the so called Russian-school-model (in \cite{GoldsheidMP-77})
and the Anderson model (in \cite{FroehlichS-83,FroehlichMSS-85}).
In the later papers the so-called \emph{multiscale analysis}
was introduced as a method for proving (Anderson) localization,
i.e.~dense pure-point spectrum, with exponentially decaying eigenfunctions, almost surely.
This method is an induction over increasing finite length scales.
On each scale one considers a restricted Hamiltonian which lives on a finite volume.
An important input for the multiscale analysis
is the induction anchor, which is called \emph{initial scale estimate}.
It is this estimate which follows from our eigenvalue bounds,
for the case  that the random Hamiltonian is the negative Laplacian
on a infinite randomly curved waveguide.
The mentioned finite volume restrictions correspond to
finite segments of the random waveguide.

In a previous paper \cite{BorisovV-11} we have studied the analogous problem for a
different type of random waveguide, which we will call for definitness
\emph{randomly shifted waveguide}, cf. Section \ref{s:comparison} for precise definitions
and a comparison of the two models. Somewhat surprisingly, although the two types or
random waveguides seem very similar, a very basic feature is different:
namely, the configuration, which produces the lowest first eigenvalue.
In one case the optimal configuration corresponds
to zero randomness, in the other to maximal randomness.
The model studied here is more difficult to analyse than the one of \cite{BorisovV-11}.
On the other hand it is also more natural, since it corresponds to the standard way how
local curvature perturbations are introduced for waveguides, cf.~e.g.~\cite{ExnerS-89}.

An interesting feature of our model in comparison to other extensively studied
random Schr\"odinger operators, e.g.~the above-mentioned Anderson-model,
is the geometric influence of the randomness.
Another model with geometric randomness, which as recently
attracted much attention, is the random displacement model, cf.{} e.g.{}
\cite{Klopp-93,BakerLS-08,GhribiK-10,KloppLNS}.
One important feature of models with geometric disorder is that
they exhibit no obvious \emph{monotonicity with respect to the individual
random variables}. This is the case for the
random displacement model,
for the randomly shifted waveguide (studied in \cite{BorisovV-11}),
as well as for the randomly curved waveguide, to which the present paper is devoted.
We refer the interested reader to the Introduction of \cite{BorisovV-11},
where  we have discussed in some
detail the challanges non-monotonicity poses in the spectral analysis of
random Hamiltonians.

There are also differences in the difficulties these models pose:
for the  random displacement model it is non-trivial to identify (all)
optimizing configurations of the random parameters.
(Here optimizing means: minimizing the first eigenvalue.)
This is easier for the two mentioned models of random waveguides,
although the optimal configurations turn out to be different in the
two cases, cf. Section \ref{s:comparison}.
On the other hand, in the random displacement model the random variables enter
the potential, i.e.~the zero order term, only, where in the waveguide models the higher order terms of
the differential operator depend on the randomness.

We would like to stress that localization for a certain type of random waveguides
was established already in \cite{KleespiesS-00}.
There the randomness enters via a variation of the width of the waveguide.
This type of perturbation leads to a quadratic form which depends monotonously on the
random variables and is thus structurally different from our model.

As already mentioned, the main challange in our analysis is the non-monotone dependence of
the Hamiltonian on the random variables.
We have thus to overcome difficulties as they are encountered in other types of random Hamiltonians
which exhibit a non-monotone dependence on the randomness.
The model of this type to which most attention was devoted
so far is the alloy-type random Schr\"odinger operators
with single site potentials of changing sign,
see e.g.{}
\cite{Klopp-95a}, \cite{Stolz-00}, \cite{Veselic-01}, \cite{Klopp-02c}, \cite{Veselic-02a},
\cite{HislopK-02}, \cite{KostrykinV-06}, \cite{KloppN-09a}, \cite{Veselic-10a}.
More recently also the discrete analog of this model was studied in
\cite{ElgartTV-10}, \cite{Veselic-10b}, \cite{ElgartTV-11},
\cite{TautenhahnV-10}, \cite{Krueger-12}, \cite{CaoE-12},
\cite{ElgartSS-12}, \cite{TautenhahnV}.
Electromagnetic Schr\"odinger operators with random magnetic field
\cite{Ueki-94}, \cite{Ueki-00a}, \cite{HislopK-02}, \cite{KloppNNN-03}, \cite{Ueki-08},
\cite{Bourgain-09}, \cite{ErdoesH-a}, \cite{ErdoesH-c}, \cite{ErdoesH-b},
as well as Laplace-Beltrami operators with random metrics \cite{LenzPV-04}, \cite{LenzPPV-08},
\cite{LenzPPV-09} exhibit a non-monotonous parameter dependence which
affects the higher order terms of the differential operator.
Another relevant model (although not defined in terms of a countable family of random parameters)
without obvious monotonicity is a random potential given by a Gaussian stochastic
field with sign-changing covariance function,
c.f.{} \cite{HupferLMW-01a}, \cite{Ueki-04}, \cite{Veselic-11}.


We already mentioned that our results in this paper and in \cite{BorisovV-11}
provide an important ingredient for the implementation of the multiscale analysis,
a way of proving spectral localization for various types of random Hamiltonians.
The second main ingredient needed to complete this proof is a Wegner type estimate, named after the paper
\cite{Wegner-81}. For the model studied by Kleespies and Stolllmann in \cite{KleespiesS-00} a Wegner estimate, and thus also spectral localization was established.
This is facilitated by the fact that in this model the randomness enters via a variation of the width of the waveguide, resulting in a monotone dependence
on the randomness.
For the model of a randomly shifted waveguide, which we studied in \cite{BorisovV-11}, it is quite likely that
the method used in \cite{GhribiHK-07} can be developed to yield {a Wegner estimate} in the weak disorder regime.
For the model studied in this paper the situation is more complicated, since the quadratic form is not a rational function of the random variables. Anyway, for both the models, the Wegner estimate is an interesting and still open problem.

%

\section{Main deterministic results}
\label{s:deterministic-results}

Let $l\geq 1$, $l \geq  a >0$  and $g\in C_0^4(\RR)$ be an function with support in $[0,a]$.
Let ${\kappa}\geqslant 0$
and $\rho:=(\rho_j)_{j\in \ZZ}$ be a sequence of reals $\rho_j \in [0, {\kappa}]$.
We introduce the function
\[
G(\cdot ,\rho)\colon \RR \to \RR, \quad G(x_1,\rho):=\sum_{j\in \ZZ} \rho_j \, g_j(x_1),
\]
where $g_j(x_1):=g(x_1-jl)$ is the translate of the function $g$ by $jl$ to the right.
Consider the curve $\Ups_\rho :=\{r_\rho(t) \mid  t\in\RR\}\subset \RR^2$
determined by the equation
\begin{equation*}
r_\rho(t):=\big(t,G(t,\rho)\big), \quad t\in\RR.
\end{equation*}
By $\nu_\rho=\nu_\rho(t)$ we denote the unit normal vector to $\Ups_\rho$ at the point $r_\rho(t)$
defined by
\begin{equation*}
\nu_\rho(t)=\left(-\frac{
G'(t,\rho)}{P_1^{1/2}(t,\rho)},\frac{1}{P_1^{1/2}(t,\rho)}
\right), \quad P_1(t,\rho):=1+\big(G'(t,\rho)\big)^2,
\end{equation*}
for $t\in\RR$. Here $'$ denotes the derivative w.r.t. $t$. We suppose that
\begin{equation}\label{2.3}
\|g\|_{C^4[0,l]}=1. 
\end{equation}
In the following when we say that a quantity does not depend on $g$ or $a$,
we mean that it can be chosen uniformly for all elements satisfying conditions  (\ref{2.3}).

In a vicinity of $\Ups_\rho$ we introduce new coordinates
$\xi:=(\xi_1,\xi_2)$ by the formula $x=x(\xi)=r_\rho(\xi_1)+\xi_2\nu_\rho(\xi_1)$.
In other words, given $\xi_1$, we take a point at $\Ups_\rho$ corresponding to $t=\xi_1$, and then shift along the direction of $\nu_\rho(\xi_1)$ by the distance $\xi_2$. We note that the coordinates $\xi$ depend on $\rho$, and are well-defined, if ${\kappa}$ is small enough.

We introduce a randomly curved waveguide  segment
\[
D_{\rho}(N,j) :=\{x: jl<\xi_1<(j+N)l,\  0<\xi_2<\pi\}.
\]
Denote by $\G_{\rho}(N,j)$ the upper and
lower part of the boundary of $D_{\rho}(N,j)$, i.e.,
\begin{equation*}
\G_{\rho}(N,j):=\{x: jl<\xi_1<(j+N)l, \xi_2=0\} \cup\{x:
jl<\xi_1<(j+N)l, \xi_2=\pi\}.
\end{equation*}
The remaining part of the boundary $\p
D_{\rho}(N,j)\setminus\overline{\G_{\rho}(N,j)}$ is
denoted by $\g_{\rho}(N,j)$. If $j=0$, we will use the
following shorthand notation:
$D_{\rho}(N):=D_{\rho}(N,0)$,
$\Gamma_{\rho}(N):=\Gamma_{\rho}(N,0)$,
$\gamma_{\rho}(N):=\gamma_{\rho}(N,0)$.
Note that the length of $D_\rho(N)$ is $L= N \cdot l$.

In order to state the deterministic lower bound on the first eigenvalue,
as a function of the parameters $\rho=(\rho_j)_{j\in\ZZ}$,
we first have to introduce a reference energy.
This will be the spectral bottom of a comparison operator.
More precisely,
we will use the symbol $\cH^{\per}(\rho)$ to denote
the negative Laplacian in $L_2(D_{\rho}(1))$ with
Dirichlet boundary condition on $\G_{\rho}(1)$ and
periodic boundary condition on $\g_{\rho}(1)$.
Note that is this situation the array $\rho=(\rho_j)_{j}$
consists just of the single value $\rho_0$.
So we identify here $\rho_0$ with $\rho$.
Let $\lambda^{\per}(\rho)$ be the lowest eigenvalue of $\cH^{\per}(\rho)$ and $\psi^{\per}=\psi^{\per}(x,\rho)$ be the associated normalized eigenfunction.
We extend the function $\psi^{\per}$ $l$-periodically w.r.t. $\xi_1$ and
denote the extension by the same symbol.

We denote by  $\cH(\rho, N, j)$
the negative Laplacian on
$L_2(D_{\rho}(N,j))$ with Dirichlet boundary conditions on
$\G_{\rho}(N,j)$ and with Robin boundary condition
\begin{equation}
\label{eq:h}
\begin{aligned}
&\left(\frac{\p}{\p x_1}-h(\cdot,{\kappa})\right)u=0\quad\text{on}
\quad \g_{\rho}(N,j), \quad \text{ where }
\\
&h(\cdot,{\kappa})\colon [0,\pi] \to \RR, \quad h(x_2,{\kappa}):=\frac{1}{\psi^{\per}(0,x_2,{\kappa})}\frac{\p\psi^{\per}}{\p x_1}(0,x_2,{\kappa}),
\end{aligned}
\end{equation}
is the logarithmic derivative in $x_2$-direction.
We shall show in Sec.~\ref{Sec:determ} that the function $h$ is well-defined.
In the following we denote for some $s \in \RR$ by $s \cdot \mathbf{1}$ the
constant sequence $\rho_j =s \ (j \in \ZZ)$.
The important point is that $\psi^{\per}(\cdot,\rho)$ is still an eigenfunction
of the operator $\cH(\rho \cdot \mathbf{1},N,j)$ with the new boundary conditions, i.e.
\[
\cH(\rho \cdot \mathbf{1},N,j) \psi^{\per} = \lambda({\rho})\psi^{\per},
\]
where we use for the restriction of $\psi^{\per}$ to the
finite waveguide segment  $D_{\rho \cdot \mathbf{1}}(N,j)$
the same symbol.
This relation was first exploited in the present context  by Mezincescu \cite{Mezincescu-87}.
For this reason we call  (\ref{eq:h}) in the sequel Mezincescu boundary conditions.
We denote the lowest eigenvalue of
$\cH(\rho, N)=\cH(\rho, N, 0)$ by  $\lambda(\rho, N)$.

\begin{theorem}\label{th:deterministic}
Assume $\alpha:= \frac{\pi^3}{32}- a>0$.
There exists a constant $\delta=\delta(a,g)>0$ independent of $\rho$, $l$, $N$
such that for
\begin{equation}\label{eq:small-coupling}
{\kappa}  \leqslant \d L^{-11/2},
\text{ and } 0 \leqslant \rho_i  \leqslant  {\kappa} \text{ for all } i \in 0, \dots N-1
\end{equation}
the estimate
\begin{equation}\label{eq:deterministic-estimate}
\lambda(\rho, N)-\lambda^{\per}({\kappa})
\geqslant
C_{lb}\frac{\kappa}{L} \sum\limits_{j=0}^{N-1}(\kappa-\rho_j)
\end{equation}
holds true with
\[
C_{lb}:=\frac{\|g'\|_{L_2(0,l)}^2}{4} \left(\frac{\pi}{2}-\frac{16a}{\pi^2}\right)
= \frac{4}{\pi^2} \, \|g'\|_{L_2(0,l)}^2\, \alpha.
\]
\end{theorem}


In view of the well known result that
curvature induces eigenvalues below the bottom
of the essential spectrum of the straight quantum waveguide, cf.\cite{ExnerS-89},
it can be considered natural that maximal coupling constants produce the lowest first
eigenvalue in the case of randomly curved waveguides


\section{Probability of low lying eigenvalues}
\label{s:probabilistic-results}

In this section we want to formulate an upper bound on the (small) probability
that a randomly perturbed, finite  waveguide segment has an eigenvalue close to the
overal possible minimal spectral value.
For this we use the deterministic results of perturbation theory,
formulated in Section \ref{s:deterministic-results}.
The probabilistic results we obtain hold in the
so-called weak disorder regime. To make this explicit, we use
in this section a global (scalar) disorder parameter:
For ${\kappa} >0$ we set
\[
 \omega_j := \rho_j/ {\kappa} \quad  (j \in \ZZ).
\]
Then we have the following equivalence for the condition used in Theorem \ref{th:deterministic}
\[
\rho_j \in [0,{\kappa}] \quad  (j \in \ZZ)
\Longleftrightarrow
\omega_j \in [0,1] \quad  (j \in \ZZ)
\]
We consider the parameters $\omega_j, j \in \ZZ$ as a sequence of non-trivial
independent, identically distributed random variables taking values in the unit interval $[0,1]$.
We denote the distribution of the single random variable $\omega_j$ by $\mu$, and by
$\PP$ the product probability measure on $\Omega:= [0,1]^{\ZZ}$ which is the distribution
of the random sequence $\omega_j, j \in \ZZ$. Of course $\Omega$ is equipped with the natural
product sigma-field.
Now the deviations of the parameters from the extremal value,
which determines the lower bound in Theorem \ref{th:deterministic},
can be rewritten as follows:
\[
\kappa-\rho_j = \kappa(1-\omega_j) = \kappa \cdot \tilde\omega_j
\text{ for all }  j \in \ZZ
\]
Here we denote by $\tilde \omega_j=(1-\omega_j)\in[0,1]$ the flipped random variables,
and the multiplication of a sequence
$\omega=(\omega_j)_{j\in \ZZ}$ with a scalar $s\in\RR$ by $s \cdot \omega$.

Now the result of Theorem \ref{th4.2} can be formulated in the new parameters
as follows: \
If $\alpha:= \frac{\pi^3}{32}-a >0$, then there exists a positive constant
$\delta$ independent of $l,N,L$ and of the sequence $\omega=(\omega_j)_{j \in \ZZ}$,
such that for all $\kappa \leq \delta L^{-11/2}$ we have
\begin{equation}
 \label{eq:lower-bound-omega-formulation}
\lambda(\kappa\cdot\omega, N)-\lambda^{\per}({\kappa})
\geqslant
C_{lb}\frac{ {\kappa}^2}{L} \sum_{j=0}^{N-1}\tilde\omega_j
\quad
\text{ where } C_{lb}= \frac{4}{\pi^2} \, \|g'\|_{L_2(0,l)}^2\, \alpha
\end{equation}
Here we denote the lowest eigenvalue of
$\cH(\kappa \cdot \omega, N)=\cH(\rho, N)$
by
$\lambda(\kappa\cdot\omega, N)=\lambda(\rho, N)$.
Recall that $\lambda^{\per}({\kappa})$
is the lowest eigenvalue of
$\cH^{\per}(\rho)$, i.e.~the negative Laplacian in $L_2(D_{\rho}(1))$ with
Dirichlet boundary condition on $\G_{\rho}(1)$ and
periodic boundary condition on $\g_{\rho}(1)$.
This is the lowest achieveable spectral value, if all perturbation parameters
are bounded by $\kappa$.

The announced bound on the probability of low-lying eigenvalues is:
\begin{theorem}
\label{th:ilse}
Let $\gamma >22, \gamma \in \NN$.
Then there exists $N_1=N_1(\gamma, \mu,l,\delta,g,a)\in (0,\infty)$,
such that for all $N \geqslant N_1$  the interval
\[
 I_N:=\left[\tilde C \, N^{-1/4},
            \delta \, l^{-11/2} N^{-\frac{11}{2\gamma}}\right]
\]
is non-empty. For $N \geqslant N_1$  and $\kappa \in I_N$, we have
\begin{equation}
\label{eq:initial}
 \PP\left(\omega \in \Omega\mid \lambda(\kappa\cdot\omega,N)-\lambda^{\per}(\kappa)
\leqslant N^{-\frac{1}{2}}\right )
\leqslant
N^{1-\frac{1}{\gamma}} \ \E^{-C_{\rm LDP}\, N^{1/\gamma}}.
\end{equation}
The constants $\tilde C=\tilde C(\mu,l,g,a)$ and  $C_{\rm LDP}=C_{\rm LDP}(\mu)>0$, as well as $N_1$
are given explicitly in Section \ref{s:probabilistic-proof}
\end{theorem}

\begin{remark}
The requirement $\kappa\in I_N$ encodes
how  our  weak-disorder regime depends on the length scale $N$.
For purpose of illustration, let us choose $\gamma =33$.
Then $\frac{1}{4}> \frac{11}{2 \cdot \gamma} = \frac{1}{6}$, thus
it is possible to choose
$\kappa = const. \, N^{-\frac{1}{4}}$,
which means that the allowed disorder regime does not shrink too fast for $ N \to \infty$.
\end{remark}

For the purposes of a localization proof, for instance using the multiscale analysis
or the fractional moment method, the following initial length scale estimate or finite volume
criterion is of interest. It can be derived from Theorem \ref{th:ilse}
using a Combes-Thomas estimate \cite{CombesT-73}. This is a  well established tool in the theory
or random (Schr\"odinger) operators. An explicit derivation adapted to the setting
of quantum waveguides can be found, e.~g., in Section 4 in \cite{BorisovV-11}.

The mentioned initial length scale estimate is a bound
on the probability that the Green's function exhibits exponential off-diagonal decay
for energies in a small interval at the bottom of the spectrum,
i.e.{} the overall minimum of the spectrum $\lambda^{\per}(\kappa)$.
For the following statement note that for $N$ large $I_N \subset [0,1]$.

\begin{corollary}\label{c:Combes-Thomas}
Let $l,a,g$, $\mu$, $I_N$, $\gamma$, $N_1$ and $C_{\rm LDP}$ be as in Theorem \ref{th:ilse}.
Let $\kappa \in I_N \cap [0,1]$, $\alpha, \beta \geqslant 2$ and set
\begin{align*}
A  &:=\{x \in D_{\kappa\cdot\omega}(N) \mid 0\leqslant  x_1\leqslant  \alpha \} ,
\\
B  &:=\{x \in D_{\kappa\cdot\omega}(N) \mid L-\beta\leqslant  x_1\leqslant  L  \} .
\end{align*}
Then there exists an absolute constant $c$ such that for any $N \geqslant N_1$
\begin{multline*}
 \PP\left(\omega \in \Omega\mid \forall  \ \lambda\leq \lambda^{\per}(\kappa)+1/(2\sqrt N)
  \colon
\left \|\chi_A (\mathcal{H}(\kappa\cdot\omega,N)-\lambda)^{-1}   \chi_B \right \|
\leqslant
\sqrt N
\ \E^{{-c \, \dist(A,B) /\sqrt N}}
\right)
\\
\geqslant
1-N^{1-\frac{1}{\gamma}} \ \E^{-c_{\rm LDP} \, N^{1/\gamma}}.
\end{multline*}
\end{corollary}


\section{Comparison with the randomly wiggled or shifted waveguide}
\label{s:comparison}

In this section we compare the results obtained in the present paper
to those of \cite{BorisovV-11}, where a related problem was studied.
There we derived deterministic and probabilistic lower bounds
on the lowest eigenvalue of finite segments of a randomly perturbed
waveguide. In contrast to the model studied here, the random geometric
perturbation was introduced by locally shifting or wiggling the waveguide.
To explain the difference between the two models we recall the
mathematical definition of the Hamiltonians studied in \cite{BorisovV-11}.


\subsection{Statement of the results of \cite{BorisovV-11} on randomly shifted waveguides}
\label{ss:randomly-shifted}

Similarly as in the model studied in this paper we
consider random quantum waveguides in $\RR^2$,
determined by the following data:
Let $(\omega_k)_{k\in \ZZ}$ be a sequence of independent, identically distributed,
non-trivial random variables with values in $[0,1]$,
$\kappa>0$ a global coupling constant, $l\geqslant 1$ the length of one (periodicity) cell of the waveguide,
and $g \in C_0^2(0,l)$ a single bump function.
The following function $G\colon\RR \times \Omega \to \RR$ determines the shape of the waveguide
\[
 G(x_1,\omega) := \sum_{k\in\ZZ} \omega_k \, g(x_1 -kl ) .
\]
Note that $\kappa G(x_1,\omega) =G(x_1,\kappa \omega)$.
For a global coupling constant $\kappa >0$ , $N\in\NN$ and  $j\in\ZZ$
we define finite segments of a infinite waveguide
\[
D_{\kappa,\omega}(N,j):=\{x \in \RR^2 \mid jl< x_1< (j+N)l , \kappa G(x_1,\omega)<x_2<\kappa G(x_1,\omega)+\pi\}.
\]
The upper and lower part of the boundary of
$D_{\kappa,\omega}(N,j)$, we denote by
\begin{align*}
\Gamma_{\kappa,\omega}(N,j):=&
\{x \in \RR^2 \mid jl< x_1< (j+N)l , x_2=\kappa G(x_1,\omega)\}
\\ \cup &
\{x \in \RR^2 \mid jl< x_1< (j+N)l , x_2=\kappa G(x_1,\omega)+\pi\}.
\end{align*}
while $\partial D_{\kappa,\omega}(N,j)\setminus
\Gamma_{\kappa,\omega}(N,j) $ is denoted by   $\gamma_{\kappa,\omega}(N,j)$.
Denote the negative Laplace operator on
$D_{\kappa,\omega}(N,j)$ with Dirichlet boundary conditions on $\Gamma_{\kappa,\omega}(N,j)$
and Neumann b.c.~on $\gamma_{\kappa,\omega}(N,j)$ by
$\mathcal{H}_{\kappa,\omega}(N,j)$, and its lowest eigenvalue
by $\lambda_{\kappa,\omega}(N,j)$.
We  use the following abbreviation:
\begin{align*}
D_{\kappa,\omega}(N)&:=D_{\kappa,\omega}(N,0),&
\Gamma_{\kappa,\omega}(N)&:=\Gamma_{\kappa,\omega}(N,0),&
\gamma_{\kappa,\omega}(N):=\gamma_{\kappa,\omega}(N,0),
\\
\mathcal{H}_{\kappa,\omega}(N)&:=\mathcal{H}_{\kappa,\omega}(N,0),&
\lambda_{\kappa,\omega}(N)&:=\lambda_{\kappa,\omega}(N,0).&
\end{align*}
Note that $D_{\kappa,\omega}(N,j) = D_{1,\kappa\omega}(N,j)$.
Since $\kappa >0$ is arbitrary we may assume without restricting the model
\begin{equation}\label{eq:normalization-shifted}
\max\{\|g\|_{C[0,l]}, \|\|g'\|_{C[0,l]}, \|g''\|_{C[0,l]}\}=1.
\end{equation}
Denote the distribution measure of $\omega_k $ by $\mu$
and by $\PP = \bigotimes\limits_{k\in\ZZ}\mu$ the
product measure  on the configuration space
$\Omega=\times_{k\in\ZZ}[0,1]$ whose elements we
denote by $\omega = (\omega_k)_{k \in\ZZ}$.\\

The minimum of the spectrum of the Laplacian on an straight waveguide segment
$\mathcal{H}_{0,\omega}(N,j)$ is equal to one, and
no operator $\mathcal{H}_{\kappa,\omega}(N,j)$  has spectrum below one.
In this sense, one is the minimal spectral value
for all random configurations.

\begin{theorem}
\label{th:ilse-shifted}
Let $\gamma >34$. Let $g$ and $\mu$ be as above and set
\begin{equation*}
\widetilde{g}:=g-\frac{1}{l}\int_0^l g(t)\di t,\quad
 c_2 =\frac{3\,\|\widetilde{g}\|_{L_2(0,l)}}{2\, l^3},
 \quad c_3 =\frac{3\,\|\widetilde{g}\|_{L_2(0,l)}^2}{5000\, l^7}.
\end{equation*}
Then there exists an initial scale $N_1$ such that if $N \geqslant N_1$  the interval
\[
 I_N:=\left[\frac{2 N^{\frac{1}{\gamma}-\frac{1}{4}}}{\EE\{\omega_k\}\sqrt{c_2}} , c_3  N^{-\frac{15}{2\gamma}} \right]
\]
is non-empty. If $N \geqslant N_1$  and $\kappa \in I_N$, then
\begin{equation}
\label{eq:initial-shifted}
 \PP\left(\omega \in \Omega\mid \lambda_{\kappa,\omega}(N)-1\leqslant N^{-\frac{1}{2}}\right )
\leqslant
N^{1-\frac{1}{\gamma}} \ \E^{-C_{\rm LDP}N^{1/\gamma}}
\end{equation}
for the constant $C_{\rm LDP}>0$ of Lemma \ref{l:LDP},
which depends only on $\mu$.
\end{theorem}

The corresponding deterministic lower bound on the first eigenvalue is:
\begin{theorem}
\label{th:deterministic-shifted}
Let $l\geqslant 1$, $g\colon \RR \to \RR$, $\widetilde{g}\in \RR$ as above, and $L:=Nl$.
For
\begin{equation}
\label{1.0c}
\kappa\left(\sum_{j=1}^N \omega_j^2\right)^{\frac{1}{2}}
\leqslant   \frac{3}{5000}  \, \|\widetilde{g}\|_{L_2(0,l)}^2  \, \frac{1}{ L^7}
\end{equation}
we have the bound
\begin{equation}
\label{eq:lower-bound-shifted}
\lambda_{\kappa,\omega}(N)-1\geqslant
\frac{3}{2}\, \|\widetilde{g}\|_{L_2(0,l)}^2 \, \frac{\kappa}{L^3}
\sum_{j=0}^{N-1} \omega_j^2
\end{equation}
\end{theorem}

\subsection{Comparison of the results on randomly shifted and curved waveguides}
\label{ss:randomly-shifted-comparison}

For definitness we will call the model of \cite{BorisovV-11}
\emph{randomly shifted waveguide}
and the model studied in the present paper \emph{randomly curved waveguide}.
\begin{enumerate}[(i)]
 \item
In the first mentioned model the configuration of parameters which
produces the minimal first eigenvalue corresponds to a straight waveguide,
where the perturbation is switched off.
In contrast to this, the configuration of the  randomly curved waveguide
yielding the minimal first eigenvalue corresponds to a periodic waveguide with
maximal possible curvature.
This difference between the two considered problems was somewhat surprising for us,
since, naively, the two models look very similar.
\item
Since for the randomly curved waveguide the bottom of the spectrum
is achieved as the perturbation is maximal,
the analysis of the model is more complicated.
For we have to perform an asymptotic expansion not around the straight waveguide,
but around the periodic waveguide $\cH(\rho \cdot \mathbf{1},N)$  with maximal allowed curvature,
which arises itself as a result of the perturbation.
\item
The smallness condition on the coupling constants
--- formulae (\ref{eq:small-coupling}) and (\ref{1.0c}) ---
and the lower bounds on the first eigenvalue
--- formulae (\ref{eq:deterministic-estimate}) and (\ref{eq:lower-bound-shifted}) ---
in Theorems \ref{th:deterministic} and \ref{th:deterministic-shifted}
exhibit a power-like, i.e. polynomial behaviour.
The exponents do not coincide, but are of the same order of magnitude.
\item
In both models  we employ a change of the variables $x\to\xi$ straightening the waveguide.
After this change the operator depends on the variables $\rho_i$.
For the randomly curved waveguide the dependence on the  $\rho_i$'s is not polynomial,
but irrational.
For the randomly shifted waveguide considered in \cite{BorisovV-11}
the relevant differential operator contains only
terms which were constant, linear or quadratic functions of the variables $\rho_i$.
\item
In Theorem~\ref{th:deterministic} we do not specify an explicit value of
$\delta$ as in the corresponding result in \cite{BorisovV-11} on the randomly shifted waveguide.
Although our technique allows us to obtain this constant  explicitly,
the required calculations are cumbersome and tedious.
The main reason why the control of constants for the randomly curved waveguide
is harder than for the randomly shifted one is
the aforementioned  irrational dependence of the perturbation on the random variables.
\end{enumerate}


\section{Probabilistic estimates: Proof of Theorem \ref{th:ilse}}
\label{s:probabilistic-proof}

We will use the following \emph{large deviations principle}
in the proof of Theorem \ref{th:ilse}. For a reference see for instance \cite{DemboZ-98}.

\begin{lemma}\label{l:LDP}
Let $\omega_k, k\in\ZZ$ be an i.i.d.~sequence of non-trivial, non-negative,
bounded random variables. Then there exists a constant $C_{\rm LDP}>0$
depending only on $\mu$ such that
\[
 \forall \ n \in \NN \ : \
\PP\left(\omega \mid \frac{1}{n} \sum_{k=1}^n \omega_k \leqslant \frac{\EE\{\omega_k\}}{2} \right ) \leqslant \E^{-C_{\rm LDP} n} .
\]
\end{lemma}

For the \emph{proof of Theorem \ref{th:ilse}}
we choose $ K \in 2 \NN$ and $\gamma \in \NN$ and set $N:= K^\gamma$.
Thus $N \in \NN$.
Set $J =N/K = K^{\gamma-1}=N^{1-\frac{1}{\gamma}}$.
Following the same ideas as in \cite{HoldenM-84,KirschSS-98a,BorisovV-11}
we decompose  the waveguide segment $D_{\kappa\cdot\omega}(N)$ into smaller parts
\begin{equation}
 \label{eq:decomposition}
\bigcup\limits^\bullet_{j=0, \dots , J-1} D_{\kappa\cdot\omega}(K,j)
\end{equation}
where $\bigcup\limits^\bullet$ denotes a disjoint union, up to a set of measure zero.
According to this  decomposition we introduce new Mezincescu boundary conditions
on the interfaces joining the shorter segments (\ref{eq:decomposition}).
As in the original paper \cite{Mezincescu-87} we conclude
that in the sense of quadratic forms
\[
\mathcal{H}(\kappa\cdot\omega,N) \geqslant \bigoplus_{j=0}^{J-1} \mathcal{H}(\kappa\cdot\omega,K,j).
\]
If we denote the lowest eigenvalue of
$\cH(\kappa\cdot\omega,N,j)$ by  $\lambda(\kappa\cdot\omega,N,j)$, it follows
\begin{equation}
\label{e:Mezincescu-bracketing}
\lambda(\kappa\cdot\omega,N) \geqslant
\min_{j=0}^{J-1} \lambda(\kappa\cdot\omega,K,j)
\end{equation}
In particular, we have the inclusion
\begin{align*}
\left\{\omega \in \Omega\mid \lambda(\kappa\cdot\omega,N)-\lambda^ {\per}(\kappa)
\leqslant N^{-\frac{1}{2}}\right \}
&\subset
\bigcup_{j=0}^{J-1} \left\{\omega \in \Omega\mid \lambda(\kappa\cdot\omega,K,j) -\lambda^ {\per}(\kappa)
\leqslant K^{-\frac{\gamma}{2}}\right \}
\end{align*}
Since the random variables $\omega_k, k \in \ZZ$ are independent and identically distributed,
we obtain
\begin{align*}
\sum_{j=0}^{J-1}
\PP\left(\omega \mid \lambda(\kappa\cdot\omega,K,j) -\lambda^ {\per}(\kappa)
\leqslant K^{-\frac{\gamma}{2}}\right )
\leqslant
N^{1-\frac{1}{\gamma}} \ \PP\left(\omega \mid \lambda(\kappa\cdot\omega,K) -\lambda^ {\per}(\kappa)
\leqslant K^{-\frac{\gamma}{2}}\right).
\end{align*}
For $\kappa \leq \delta L^{-11/2}$  we have
$\lambda(\kappa\cdot\omega, K)-\lambda^{\per}({\kappa})
\geqslant
C_{lb}\frac{ {\kappa}^2}{l \,K} \sum_{j=0}^{K-1}\tilde\omega_j $
and thus the inclusion
\begin{align*}
 \left\{
\omega \mid \lambda(\kappa\cdot\omega,K) -\lambda^ {\per}(\kappa)\leqslant K^{-\frac{\gamma}{2}}\right\}
& \subset
\left\{ \omega \mid C_{lb}\frac{ {\kappa}^2}{l\cdot K} \sum_{j=0}^{K-1}\tilde\omega_j
\leqslant K^{-\frac{\gamma}{2}}\right\}
\\ =
\left\{ \omega \mid \frac{1}{K} \sum_{j=0}^{K-1}\tilde\omega_j
\leqslant
\frac{l}{{C_{lb}\kappa}^2}  K^{-\frac{\gamma}{2}}\right\}
\end{align*}
Denote by $\EE\{\tilde\omega_0\}=\EE\{1-\omega_0\}$ the expectation value of (any) $\tilde\omega_k$.
Choose now $\kappa$ such that
\begin{equation}\label{e:kappa-conditions}
\frac{l}{{C_{lb}\kappa}^2}  \, K^{-\frac{\gamma}{2}} \leqslant \frac{\EE\{\tilde\omega_0\}}{2}
\quad \text{ i.e. } \quad
 \tilde C  K^{-\frac{\gamma}{4}}\leqslant \kappa,
\text{ where }
\tilde C  :=\sqrt{\frac{2 \, l}{C_{lb} \EE\{\tilde\omega_k\}}}  .
\end{equation}
On the other we also need to satisfy the restriction
$\kappa \leq \delta (l \, K)^{-11/2} =\delta l^{-11/2} \, N^{-11/(2\gamma)}$.
The upper and the lower bound for $\kappa$ can be reconciled if
$\gamma > 22$ and
\[
K \geqslant K_1 := \left( \frac{\tilde C \, l^{11/2} }{\delta} \right)^{\frac{4}{\gamma -22}}.
\]
The last inequality is equivalent to
\[
N\geqslant N_1:= \left(\frac{2 \, l^{12}}{C_{lb} \EE\{\tilde\omega_0\}\delta^2 }  \right )^{\frac{2\gamma}{\gamma-22}}
\]
For $\kappa$ satisfying (\ref{e:kappa-conditions}) we
are able to apply the large deviations principle of Lemma \ref{l:LDP} and thus obtain
\begin{align*}
\PP\left\{ \omega \mid \frac{1}{K} \sum_{j=0}^{K-1}\tilde\omega_j
\leqslant
\frac{l}{{C_{lb}\kappa}^2}  K^{-\frac{\gamma}{2}}\right\}
\leqslant
\PP\left\{ \omega \mid \frac{1}{K} \sum_{j=0}^{K-1}\tilde\omega_j
\leqslant \frac{\EE\{\tilde\omega_0\}}{2}\right\}
\leqslant \E^{-C_{\rm LDP}K}
\end{align*}
for $K\geqslant K_1$. Hence, for $N\geqslant N_1$
\[
\PP
\left\{\omega \mid \lambda(\kappa\cdot\omega,N)-\lambda^ {\per}(\kappa)
\leqslant N^{-\frac{1}{2}}\right \}
\leqslant
N^{1-\frac{1}{\gamma}} \ \E^{-C_{\rm LDP} N^{1/\gamma}}.
\]


\section{New parameters corresponding to a linearization around the optimal configuration}
\label{s:deterministic-reformulation}

Here we reformulate our main deterministic result, i.e.~Theorem~\ref{th:deterministic}
in a different parametrization, which reflects the perturbation theoretic
proof.
In what follows we want to study waveguide segments which correspond
to configurations $\rho$ where all coefficients $\rho_i$ are close to ${\kappa}$.
In other words, we perform an asymptotic analysis around a linearisation
of the operator family, not around the point $\rho =0 \cdot \textbf{1}$ in parameter space,
but rather around $\rho =\kappa\cdot \textbf{1}$.
Thus the natural parameters for perturbation theory are
\[
 \eps:=(\epsilon_j)_{j\in \ZZ}, \eps_j:={\kappa}-\rho_j
\]
In this and the following sections we
denote by $\cH_{{\kappa}}(\eps, N)$ the negative Laplacian on
$L_2(D_{\rho}(N))$ with Dirichlet boundary condition on
$\G_{\rho}(N)$ and with Robin boundary condition, cf.~(\ref{eq:h}),
\begin{equation}
\label{eq:Mezincescu-bc}
\begin{aligned}
&\left(\frac{\p}{\p x_1}-h(\cdot,{\kappa})\right)u=0\quad\text{on}
\quad \g_{\rho}(N),
\\
&\text{where } \ h(\cdot,{\kappa})\colon [0,\pi] \to \RR,
\\
&\quad h(x_2,{\kappa}):=\frac{1}{\psi^{\per}(0,x_2,{\kappa})}\frac{\p\psi^{\per}}{\p x_1}(0,x_2,{\kappa}).
\end{aligned}
\end{equation}
is the logarithmic derivative in $x_2$-direction.
Recall that $\psi^{\per}(\cdot, {\kappa})$ still satisfies the eigenvalue equation
of the operator $ \cH_{{\kappa}}(0, N) $ with the new boundary conditions, i.e.
\[
 \cH_{{\kappa}}(0, N) \psi^{\per} = \lambda({\kappa})\psi^{\per}
\]
In the following we denote the lowest eigenvalue of $\cH_{{\kappa}}(\eps, N)$ by  $\lambda(\eps,{\kappa})$
and suppress here the $N$-dependence in this notation. Here is the announced reformulation of Theorem~\ref{th:deterministic}.

\begin{theorem}\label{th4.2}
Suppose $a<\frac{\pi^3}{32}$.
There exists a constant $\delta>0$  independent of $\rho$, $l$, $N$ 
such that for
\begin{equation}\label{2.1}
{\kappa}  \leqslant \d L^{-11/2},
\text{ and } \rho_i  \leqslant  {\kappa} \text{ for all } i \in 0, \dots N-1
\end{equation}
the estimate
\begin{equation}\label{2.2}
\lambda(\eps,{\kappa})-\lambda^{\per}({\kappa})
\geqslant  \frac{\|g'\|_{L_2(0,l)}^2}{4L} \left(\frac{\pi}{2}-\frac{16a}{\pi^2}\right) {\kappa} \sum\limits_{j=0}^{N-1}\eps_j
\end{equation}
holds true.
The value of $\delta$ depends on $g$ and $a$.
\end{theorem}


\section{Deterministic lower bounds: Preliminaries\label{Sec:determ}}
\label{s:deterministic-preliminaries}

In this section we present certain less technical preliminaries which
are necessary for the setup of perturbation theory and the proof of Theorem \ref{th4.2}.
First we introduce the necessary notation.
Given $\rho$, we define the operator
\begin{equation*}
\cT(\rho)\colon L_2(D_{\rho}(N))\to L_2(D_{0}(N)),\quad (\cT(\rho)u)(\xi):=u(x(\xi)).
\end{equation*}
corresponding to the change of the variables $x\mapsto\xi$, where $\xi$ is associated with $\rho$. If $\rho_i={\kappa}$ for all $i$, we will write shortly $\cT({\kappa})$ instead of
$\cT({\kappa},{\kappa},\ldots)$. Let us study in more detail the change $x\mapsto \xi$. It is easy to check that
\begin{gather*}
\nabla_x=\mathrm{A}\nabla_\xi,
\\
\mathrm{A}^{-1}:=
\begin{pmatrix}
\frac{\p x_1}{\p \xi_1} & \frac{\p x_2}{\p \xi_1}
\\
\frac{\p x_1}{\p \xi_2} & \frac{\p x_2}{\p \xi_2}
\end{pmatrix}=
\begin{pmatrix}
1-\xi_2 (G'P_1^{-1/2})' & G'+\xi_2 (P_1^{-1/2})'
\\
-G' P_1^{-1/2} & P_1^{-1/2}
\end{pmatrix}.
\end{gather*}
We set
\[
P_2(\xi,\rho):=1-\xi_2 K(\xi_1,\rho), \quad
K(\xi_1,\rho):=G''(\xi_1,\rho)P_1^{-3/2}(\xi_1,\rho)
\]
and obtain by direct calculation
\begin{gather}   
P_3:=\Det^{-1} \mathrm{A}=P_1^{-1/2}(1-\xi_2 G'' P_1^{-1/2}+G'^2)=P_1^{1/2} P_2,
\end{gather}
where $'$ denotes the derivative w.r.t. $\xi_1$. Hence,
\begin{equation}\label{4.2}
\|u\|_{L_2(D_\rho(N))}^2=\la \cT(\rho) u, P_3\cT(\rho) u\ra_{L_2(D_{0}(N))}.
\end{equation}
Thus the operator
\begin{equation*}
\sqrt{P_3}\cT(\rho)\colon L_2(D_{\rho}(N))\to L_2(D_{0}(N)),\quad (\sqrt{P_3}\cT(\rho)u)(\xi):=\sqrt{P_3}u(x(\xi))
\end{equation*}
is unitary.
We observe that $P_i(\xi, \rho)=1+\Odr({\kappa})$, $i=1,2,3$, for $\rho_j \leq {\kappa}$.
Thus for sufficiently small values of ${\kappa}$, non of these functions vanish.
One can calculate directly that
\begin{align}
& P_3 \mathrm{A}^* \mathrm{A}=
\begin{pmatrix}
P_3^{-1} & 0
\\
0 & P_3
\end{pmatrix}\label{4.5}
\\
&
\begin{aligned}
\|\nabla_x u\|_{L_2(D_\rho(N))}^2=&\la\mathrm{A}\nabla_{\xi} \cT(\rho) u, P_3 \mathrm{A}\nabla_{\xi} \cT(\rho) u\ra_{L_2(D_{0}(N))}
\\
=&\left\|P_3^{-1/2} \frac{\p \cT(\rho) u}{\p\xi_1} \right\|_{L_2(D_0(N))}^2
+\left\|P_3^{1/2} \frac{\p \cT(\rho) u}{\p\xi_2} \right\|_{L_2(D_0(N))}^2,
\end{aligned}
\label{4.3}
\\
&\D_x=\frac{1}{P_3} \left(\frac{\p}{\p\xi_1} \frac{1}{P_3} \frac{\p}{\p\xi_1}+\frac{\p}{\p\xi_2} P_3\frac{\p}{\p\xi_2}\right)=
\D_\xi-\mathcal{Q}_\rho=\cH^{\per}(0)+\mathcal{Q}_{\rho},\label{4.4}
\\
&\mathcal{Q}_\rho:=Q_{11}\frac{\p^2}{\p \xi_1^2}+Q_{1}\frac{\p}{\p\xi_1}+Q_{2}\frac{\p}{\p\xi_2},\label{2.0}
\\
&\mathcal{Q}_{11}(\xi,\rho):=1-\frac{1}{P_3^2}
=
-\frac{2\xi_2 K(\xi_1,\rho)-\xi_2^2 K^2(\xi_1,\rho)}{P_2^2(\xi,\rho)}+\frac{\big(G'(\xi_1,\rho)\big)^2} {P_1(\xi_1,\rho) P_2^2(\xi,\rho)},\nonumber
\\
&\mathcal{Q}_1(\xi,\rho):=-\frac{1}{P_3}\frac{\p}{\p\xi_1}\frac{1}{P_3}
=
-\frac{\xi_2  K'(\xi_1,\rho)}{P_1(\xi,\rho)P_2^3(\xi_1,\rho)}+\frac{ G'(\xi_1,\rho) K(\xi_1,\rho)}{P_1^{1/2}(\xi_1,\rho)P_2^2(\xi_1,\rho)},
\nonumber
\\
&\mathcal{Q}_2(\xi,\rho):=-\frac{1}{P_3}\frac{\p P_3}{\p\xi_2}
=
\frac{K(\xi_1,\rho)}{P_2(\xi,\rho)}.
\nonumber
\end{align}
Under the standing assumption (\ref{2.3}) we can find a constant
$C$ independent of $l$, $a$, $g$, $N$, and ${\kappa}$ such that
\begin{equation}\label{4.9}
\|\mathcal{Q}_{11}\|_{C(\overline{D_0(N)})}+\|\mathcal{Q}_1\|_{C(\overline{D_0(N)})} +\|\mathcal{Q}_2\|_{C(\overline{D_0(N)})}
\leqslant C{\kappa}
\end{equation}
uniformly for all configurations satisfying $ \rho_i  \leqslant  {\kappa} \text{ for all } i \in 0, \dots N-1$.
\medskip

Our final aim is to have a lower bound on the movement of eigenvalues under a perturbation.
However, first we need some rough a-priori   information about the position of eigenvalues.
This amounts of showing that the relevant eigenvalue is not too far from the one of a simple
 reference operator.
So an upper bound on the distance is in question.
For this preliminary considerations it will be sufficient to consider the operator on a waveguide segment of
length $l$, i.e. the case $N=1$. Before we state the next Lemma,
we collect some simplifications
which will be useful here and in subsequent considerations.

\begin{remark}[Simplified expressions for the case $N=1$]
Here we consider the case $N=1$, i.e. the operator  $ \cH^{\per} ({\kappa})$
on the unit cell $D_{{\kappa}}(1)$. Note that in this situation there is only one
coupling constant $\rho_0$ in  the game; for simplicity we will write here $\rho=\rho_0$
and similarly $g(\xi_1)=g_0(\xi_1)$.
Furthermore the following functions take on a simple, explicit form:
\begin{align*}
P_1(\xi_1,\rho)  &=1+\big(G'(\xi_1,\rho)\big)^2= 1+\big(\rho g'(\xi_1)\big)^2
\\
K(\xi_1,\rho)  &=\rho g''(\xi_1) P_1^{-3/2}(\xi_1,\rho)
=\rho g''(\xi_1) \big( 1+(\rho g'(\xi_1))^2\big)^{-3/2}
\\
P_2 (\xi,\rho)  &=1-\xi_2 K(\xi_1,\rho)= 1-\xi_2 \rho g''(\xi_1) P_1^{-3/2}(\xi_1,\rho)
\\
P_3(\xi,\rho)   &=P_1^{1/2} P_2= P_1^{1/2}(\xi_1,\rho) -\xi_2 \rho g''(\xi_1) P_1^{-1}(\xi_1,\rho)
\\
\mathcal{Q}_2(\xi,\rho)&
=\frac{K(\xi_1,\rho)}{P_2(\xi,\rho)}
=\frac{\rho g''(\xi_1) }{P_2(\xi,\rho)P_1^{3/2}(\xi_1,\rho)}
\end{align*}
In the following we will make use of the following error term estimates on the functions above:
\begin{align*}
P_1^{1/2}(\xi_1,\rho)  &=1+\frac{\rho^2}{2}  g'(\xi_1)^2 +\Odr(\rho^4),
\quad
P_1^{-1}(\xi_1,\rho)   =1-\rho^2  g'(\xi_1)^2 +\Odr(\rho^4),
\\
P_3(\xi,\rho)  &= P_1^{1/2}(\xi_1,\rho)   -  \rho\xi_2 g'' (\xi_1) P_1^{-1}(\xi_1,\rho)
\\
& = 1 -\rho\xi_2 g''(\xi_1)+\frac{\rho^2}{2}  g'(\xi_1)^2  +\rho^3\xi_2 g'' (\xi_1)  g'(\xi_1)^2  +\Odr(\rho^4).
\end{align*}
in particular $\mathcal{Q}_\rho=\Odr(\rho)$.
 \end{remark}

\begin{lemma}\label{lm4.1}
There is a constant $c$ independent of $a$, $l$, $g$ such that for $\rho<c$ the eigenpair $\lambda^{\per}(\rho)$,
$\psi^{\per}(\cdot,\rho)$  of the operator $ \cH^{\per} (\rho)$
on the unit cell $D_{\rho}(1) $ satisfies the relations
\begin{equation*}
|\lambda^{\per}(\rho)-1|  \leqslant   C\rho^2,\quad
\|\cT(\rho)\psi^{\per}(\cdot,\rho)-\psi_0^{\per}-
\rho\psi_1^{\per}\|_{C^3(\overline{D_{0}(1)})}  \leqslant C\rho^2,
\end{equation*}
where the constant $C$ is independent of $\rho$, $a$, $l$, and $g$. Here
$\psi_0^{\per}(\xi):=\sqrt{2/(\pi l)}\sin\xi_2$ is the ground state of $\cH^{\per}(0)$,
and $\psi_1^{\per}$ is
orthogonal to $\psi_0^{\per}$ in $L_2(D_{0}(1))$ and solves the equation
\begin{equation}\label{4.34}
(\cH^{\per}(0)-1)\psi_1^{\per}=-g''\frac{\p\psi_0^{\per}}{\p\xi_2}.
\end{equation}
\end{lemma}

\begin{remark}\label{rm3.1}
The fact that the function $\psi_0^{\per}$ depends only on the variable $\xi_2$
leads to several simplifications.
On the one hand, $\mathcal{Q}_{\rho} \psi_0^{\per}=
\mathcal{Q}_{2}(\rho) \frac{\partial}{\partial \xi_2} \psi_0^{\per}$.
On the other, integrals involving the functions $g$  and $\psi_0$
can be often simplified by separation of variables, e.g.:
\[
 \la g''\frac{\p\psi_0^{\per}}{\p\xi_2} , \psi_0^{\per} \ra_{L_2(D_{0}(1))}
= \frac{2}{l \pi} \int_0^l d\xi_1 \, g''(\xi_1) \, \int_0^\pi d\xi_2 \, \cos(\xi_2) \sin(\xi_2)
=0.
\]
In the following this will be used repeatedly.
\end{remark}

\begin{proof}
The operator $\cT(\rho)\cH^{\per}(\rho)\cT^{-1}(\rho)$
can be regarded as a perturbation of $\cH^{\per}(0)$, since
$\cT(\rho)\cH^{\per}(\rho)\cT(\rho)^{-1}= \cH^{\per}(0)+\mathcal{Q}_{\rho}$.
Note that $\mathcal{Q}_{\rho}$ is relatively bounded w.r.t. the operator $\cH^{\per}(0)$.
This follows, since the boundary conditions of the functions in the domain of
$\cH^{\per}(0)$ imply that the boundary integral terms resulting from partial
integration vanish. The bound (\ref{4.9}) implies that for sufficiently small $\rho$,
$\mathcal{Q}_{\rho}$ is relatively bounded w.r.t. $\cH^{\per}(0)$
with relative bound strictly smaller than one.
Consequentyl the sum $\cH^{\per}(0)+\mathcal{Q}_{\rho}$ is a closed operator on the domain of
$\cH^{\per}(0)$.
For a $\psi$ in this domain and every $\phi\in L^2(D_{0}(1))$
the function $\rho \mapsto \la \phi, \mathcal{Q}_{\rho} \psi\ra$ is holomorphic for
$\rho$ in a small complex neighbourhood of zero. Thus
$\rho \mapsto \cH^{\per}(0)+\mathcal{Q}_{\rho}$ is a holomorphic family of operators of type (A)
in the terminology of \cite[VII.\S2]{Kato-66}.
Note that one is the lowest eigenvalue of $\cH^{\per}(0)$
and $\psi_0^{\per}$ is the associated normalized eigenfunction.
Thus, e.g., Theorem II.5.11 of  \cite{Kato-66} implies
that there is a constant $C$  such that
\begin{equation}\label{4.35}
\begin{aligned}
&\lambda^{\per}(\rho)=1+\rho\lambda_1^{\per}+\rho^2\lambda_2^{\per}(\rho), && |\lambda_2^{\per}(\rho)|\leqslant C,
\\
&\cT(\rho)\psi^{\per}(\cdot,\rho)=\psi_0^{\per}+\rho\psi_1^{\per}+\rho^2\psi_2^{\per}(\cdot,\rho) && \|\psi_2^{\per}\|_{L_2(D_0(1))}\leqslant C,
\end{aligned}
\end{equation}
where $\lambda_1^{\per} \in \RR$ and $\psi_1^{\per}\colon D_0(1)\to \RR $ are independent of $\rho$.
The constant $C$ is independent of $\rho$,  $a$, $l$, and $g$ since the
perturbation $\mathcal{Q}_{\rho}$ is uniformly bounded in these parameters,
provided the normalization condition (\ref{2.3}) holds.

We substitute (\ref{4.35})
 into the eigenvalue equation
$ \cH^{\per} (\rho)\psi^{\per}(\cdot,\rho)=  \lambda^{\per}(\rho)\psi^{\per}(\cdot,\rho)$
and obtain
\begin{equation}\label{3.7a}
\begin{aligned}
\cT(\rho)\big(\cH^{\per}(0)+\mathcal{Q}_{\rho}\big)
(\psi_0^{\per}+\rho\psi_1^{\per}+\rho^2\psi_2^{\per})
&=\cH^{\per} (\rho)\psi^{\per}(\cdot,\rho)=  \lambda^{\per}(\rho)\psi^{\per}(\cdot,\rho)
\\
&=  \cT(\rho)\left(1+\rho\lambda_1^{\per}+\rho^2\lambda_2^{\per}\right)
\left(\psi_0^{\per}+\rho\psi_1^{\per}+\rho^2\psi_2^{\per}\right).
\end{aligned}
\end{equation}
Thus,
\begin{align*}
& \cH^{\per}(0)\psi_0^{\per}+\rho\cH^{\per}(0)\psi_1^{\per}
+ \mathcal{Q}_{\rho} \psi_0^{\per}+\rho \mathcal{Q}_{\rho}\psi_1^{\per}+\Odr(\rho^2)
=
\psi_0^{\per}+\rho\psi_1^{\per}
+ \rho\lambda_1^{\per} \psi_0^{\per}+\Odr(\rho^2).
\end{align*}
We want to isolate the terms which are linear in the perturbation parameter $\rho$.
For this reason we substract the eigenvalue equation
$\cH^{\per}(0)\psi_0^{\per}=\psi_0^{\per}$
for the unperturbed operator,
\begin{align*}
 \rho\cH^{\per}(0)\psi_1^{\per} + \mathcal{Q}_{\rho} \psi_0^{\per}+\rho \mathcal{Q}_{\rho}\psi_1^{\per}
&=
 \rho\cH^{\per}(0)\psi_1^{\per} + \mathcal{Q}_{2}(\xi,\rho) \frac{\p\psi_0^{\per}}{\p \xi_2}+\Odr(\rho^2)
\\
&=
\rho\psi_1^{\per} + \rho\lambda_1^{\per} \psi_0^{\per}+\Odr(\rho^2),
\end{align*}
divide by $\rho$, and finaly take  the limit $\rho\to 0$
resulting in:
\begin{equation*}
(\cH^{\per}(0)-1)\psi_1^{\per}=-g''\frac{\p\psi_0^{\per}}{\p\xi_2}+\lambda_1^{\per} \psi_0^{\per}.
\end{equation*}
This equation is solvable, if and only if the right hand side is
orthogonal to $\psi_0^{\per}$ in $L_2(D_{0}(1))$, i.e.,
\begin{equation*}
\lambda_1^{\per}=\lambda_1^{\per}\la\psi_0^{\per},\psi_0^{\per} \ra_{L_2(D_{0}(1))}
= \la g''\frac{\p\psi_0^{\per}}{\p\xi_2} , \psi_0^{\per} \ra_{L_2(D_{0}(1))}
=0
\end{equation*}
This implies (\ref{4.34}). The
orthogonality condition for $\psi_1^{\per}$ is implied by the identities
\begin{align*}
1=&\|\psi^{\per}\|_{L_2(D_\rho(1))}^2=\la\cT(\rho)\psi^{\per},
P_1^{1/2}P_2\cT(\rho)\psi^{\per}\ra_{L_2(D_0(1))}
\\
=&\la\psi_0^{\per},P_3 \psi_0^{\per}\ra_{L_2(D_{0}(1))}
+2\rho\la\psi_0^{\per}, P_3 \psi_1^{\per}\ra_{L_2(D_0(1))}+\Odr(\rho^2)
\\
=&\la\psi_0^{\per},(1-\rho\xi_2g_1'')
\psi_0^{\per}\ra_{L_2(D_{0}(1))}
+2\rho\la\psi_0^{\per}, \psi_1^{\per}\ra_{L_2(D_0(1))}+\Odr(\rho^2)
\\
=&1+2\rho\la\psi_0^{\per}, \psi_1^{\per}\ra_{L_2(D_0(1))}+\Odr(\rho^2)
\end{align*}
since $ \int_0^ld\xi_1 g''(\xi_1)=0$.

The functions $\cT(\rho)\psi^{\per}$ and $\psi_1^{\per}$ belong to  ${W_2}^1(D_0(1))$.
By the standard smoothness improving theorems (see, for instance, \cite[Ch. IV, Sec. 2]{Mikhajlov-76}) and by the smoothness of the function $g$ we obtain that $\cT(\rho)\psi^{\per}$ and $\psi_1^{\per}$ are also elements of ${W_2}^4(D_0(1))$.
In view of the embedding ${W_2}^4(D_0(1))\subset C^2(\overline{D_0(1)})$ \cite[Ch I\!I\!I, Sec. 6]{Mikhajlov-76} the functions $\cT(\rho)\psi^{\per}$ and $\psi_1^{\per}$ are the classical solutions to the eigenvalue problem for $\psi_1^{\per}$ and to (\ref{4.34}). By applying Schauder estimates (see \cite[Ch. I\!I\!I, Sec. 1-3]{LadyzhenskayaU-73}) we conclude that $\cT(\rho)\psi^{\per}$ and $\psi_1^{\per}$ belong to $C^3(\overline{D_0(1)})$.

It remains to prove that the asymptotics (\ref{4.35}) for
$\cT(\rho)\psi^{\per}$ holds true also in
$C^3(\overline{D_0(1)})$-norm. In order to do it, by (\ref{3.7a}) we
write first the equation for $\psi_2^{\per}$
\begin{align*}
\cH^{\per}(0)\psi_2^{\per}=&(-\cQ_\rho+\lambda^{\per})\psi_2^{\per}
\\
&+\cT^{-1}(\rho) \left[\lambda_2^{\per}(\psi_0^{\per}+\rho\psi_1^{\per}) + \lambda_1^{\per}\psi_1^{\per} +\rho^{-1} \cQ_\rho\psi_1^{\per} +\rho^{-2}\left( \cQ_\rho-\rho g''\frac{\p}{\p \xi_2}\right)\psi_0^{\per}\right].
\end{align*}
As above we again apply smoothness improving theorems from \cite{Mikhajlov-76}, \cite{LadyzhenskayaU-73} and obtain the uniform in $\rho$ estimate
\begin{equation*}
\|\psi_2^{\per}\|_{C^3(\overline{D_0(1)})}\leqslant C,
\end{equation*}
which yields the desired asymptotics for $\psi^{\per}$.
\end{proof}

Since $\psi^{\per}$ is the ground state, it is positive in $D_{\rho}(1)$. Hence, $h$ in  (\ref{eq:Mezincescu-bc}) is well-defined.
Let us check its behaviour near the Dirichlet boundaries.
It follows from the definition of $\cT(\rho)$ that
\begin{equation*}
\big(\cT(\rho)\psi^{\per}\big)(0,\xi_2,\rho)=\psi^{\per}(0,\xi_2,\rho).
\end{equation*}
In view of the smoothness of $\psi^{\per}$ we can apply Taylor formula as $\xi_2\to+0$,
\begin{equation*}
\psi^{\per}(0,\xi_2,\rho)=\xi_2 \frac{\p\cT(\rho)\psi^{\per}}{\p\xi_2}\bigg|_{\xi=0}+\xi_2^2 \frac{\p^2\cT(\rho)\psi^{\per}}{\p\xi_2^2}\bigg|_{\genfrac{}{}{0 pt}{}{\xi_1=0\ }{\xi_2=\widetilde{\xi}_2}},
\end{equation*}
for some  $\widetilde{\xi}_2\in[0,\xi_2]$. Here we have also employed the Dirichlet condition for $\psi^{\per}$ at $\xi_2=0$. We substitute the asymptotics for $\cT(\rho)\psi^{\per}$  from Lemma~\ref{lm4.1} into the last formula and arrive at
\begin{equation*}
\psi^{\per}(x_1,0,\rho)-\sqrt{\frac{2}{\pi l}} x_2=\Odr(x_2\rho+x_2^2)\quad \text{as}\quad x_2\to+0,\quad \rho\to0.
\end{equation*}
In the same way we obtain
\begin{equation*}
\psi^{\per}(x_1,0,\rho)-\sqrt{\frac{2}{\pi l}} x_2=\Odr\big((\pi-x_2)\rho+(\pi-x_2)^2\big)\quad \text{as}\quad x_2\to+0,\quad \rho\to0.
\end{equation*}
%
%
%
Therefore, for small $\rho$ this function satisfies the estimate
\begin{equation}\label{4.36}
\|h\|_{C^2[0,\pi]}  \leqslant   C\rho,
\end{equation}
where the constant $C$ is independent of $l$, $\rho$, and $g$.

\medskip

After this preliminary analysis of the operator on the unit cell, i.e. a waveguide segment of length $l$,
we turn now to waveguide segments composed of $N$ cells.
In what follows we denote all numerical constants by $C$.
In different formulas this symbol  will denote different positive numbers independent of $L$, $N$, $l$, $\rho$, $a$, and
$g$.   Let $\cH_{0}(0, N)$ be the operator $\cH_{{\kappa}}(0, N)$ taken for ${\kappa}=0$.
In fact, $\cH_0(0, N)$ is the Laplacian on $D_0(N)$
with Dirichlet condition on $\G_0(N)$ and with Neumann
condition on $\g_0(N)$. Note that the ground state $\psi_0^{\per}$
of $\cH_0(0, N)$ is a function of $\xi_2$ only. Thus the function $H$ defining the
Mezincescu b.c. vanishes.

The lowest eigenvalue of $\cH_{{\kappa}}(0, N)$ is $\lambda^{\per}({\kappa})$; the corresponding eigenfunction is
the function $\psi^{\per}$. Let $\widehat{\lambda}^{\per}({\kappa})$ be the
second eigenvalue of $\cH_{{\kappa}}(0, N)$. For the purposes of perturbation theory we need
to establish a minimal distance between the two lowest eigenvalues $\lambda^{\per}({\kappa})$ and $\widehat \lambda^{\per}({\kappa})$
For this  purpose we need
\begin{lemma}\label{lm4.2}
Under the assumption (\ref{2.3})
there exists a constant $C$ independent of $l,N, a, g$ such that the eigenvalue $\widehat{\l}^{p}$ obeys the inequality
\begin{equation*}
|\widehat{\l}^{p}({\kappa})-1-4\pi^2 L^{-2}|  \leqslant   C{\kappa}.
\end{equation*}
\end{lemma}

\begin{proof}
The quadratic form associated with $\cH_{{\kappa}}(0, N)$
reads as follows
\begin{equation*}
\la\cH_{{\kappa}}(0, N)u,u\ra_{L_2(D_{{\kappa}}(N))}=\|\nabla u \|_{L_2(D_{{\kappa}}(N))}^2 +\int\limits_{x_1=0}
h|u|^2\di x_2
-\int\limits_{x_1=L}
h|u|^2\di x_2.
\end{equation*}
Denoting $v:=\cT({\kappa}) u$, by (\ref{4.3}), (\ref{4.5}) we rewrite the form as follows
\begin{align*}
\la\cH_{{\kappa}}(0, N)u,u\ra_{L_2(D_{{\kappa}}(N))}
=&\left\|P_3^{-1/2} \frac{\p v}{\p\xi_1} \right\|_{L_2(D_0(N))}^2
+\left\|P_3^{1/2} \frac{\p v}{\p\xi_2} \right\|_{L_2(D_0(N))}^2
\\
&+\int_0^\pi
h(\xi_2,{\kappa})(|v(0,\xi_2)|^2-|v(L,\xi_2)|^2)\di \xi_2.
\\
=&
\la\cH_{0}(0, N)v,v\ra_{L_2(D_{0}(N))}
+
\la\cQ_{{\kappa}}v,v\ra_{L_2(D_{0}(N))}
\end{align*}
Since the two lowest eigenvalues of $\cH_{0}(0, N)$ are $1$ and $1+4\pi^2 L^{-2}$,
it is sufficient to estimate the quadratic form of $\cQ_{{\kappa}}$.
It will be convenient to use the abbreviation
$\la\cdot,\cdot\ra_{0}:=\la\cdot,\cdot\ra_{L_2(D_{0}(N))}$ and
$\|\cdot\|_{0}:=\|\cdot\|_{L_2(D_{0}(N))}$, in the sequel.
The estimate (\ref{4.36}) and standard embedding theorems yield
\begin{equation}\label{4.8}
\left|\int\limits_0^{\pi} h(\xi,{\kappa})(|v(0,\xi_2)|^2-|v(L,\xi_2)|^2)\di \xi_2\right|
\leqslant C{\kappa}\|\nabla_\xi v\|_{0}^2.
\end{equation}
Thus we obtain
\begin{align*}
\la\cQ_{{\kappa}}v,v\ra_{L_2(D_{0}(N))}
=&
\left\la \frac{\p  v}{\p\xi_1}, \left(1-P_3^{-1}\right) \frac{\p  v}{\p\xi_1}\right\ra_0
+
\left\la \frac{\p  v}{\p\xi_2}, \left(1-P_3\right) \frac{\p  v}{\p\xi_2}\right\ra_0
\\
&+
\int\limits_0^{\pi} h(\xi,{\kappa})(|v(0,\xi_2)|^2-|v(L,\xi_2)|^2)\di \xi_2
\\
=&
\left\la \frac{\p  v}{\p\xi_1}, \left({{\kappa}} \xi_2 g'' +\Odr({{\kappa}}^2)\right) \frac{\p  v}{\p\xi_1}\right\ra_0
+
\left\la \frac{\p  v}{\p\xi_2}, \left(-{{\kappa}} \xi_2 g'' +\Odr({{\kappa}}^2)\right) \frac{\p  v}{\p\xi_2}\right\ra_0
\\
&+ \Odr \left ({{\kappa}} \|\nabla v\|_0^2\right)
\\
=&
{\kappa}\left\la \frac{\p  v}{\p\xi_1}, \xi_2 g'' \frac{\p  v}{\p\xi_1}\right\ra_0
-
{\kappa}\left\la \frac{\p  v}{\p\xi_2}, \xi_2 g'' \frac{\p  v}{\p\xi_2}\right\ra_0
+ \Odr \left ({{\kappa}} \|\nabla v\|_0^2\right)
\\
=& \Odr \left ({{\kappa}} \|\nabla v\|_0^2\right)
\end{align*}
Thus we have the estimate
\begin{equation*}
\la\cH_{0}(0, N)v,v\ra_{0} + \la\cQ_{{\kappa}}v,v\ra_{0}
= \left(1 +\Odr ({\kappa} ) \right) \la\cH_{0}(0, N)v,v\ra_{0}
\end{equation*}
Now we apply the minimax principle and obtain
\begin{align*}
\widehat{\l}^{p}({\kappa})
&
=\max_{w \in L_2(D_0(N))} \   \min_{v \in W_2^1(D_0(N)), v\perp w, \|v\|=1}
\la(\cH_{0}(0, N)+\cQ_{{\kappa}})v,v\ra_{0}
\\
&= \left(1 +\Odr ({\kappa} ) \right)
\max_{w \in L_2(D_0(N))}  \   \min_{v \in W_2^1(D_0(N)), v\perp w, \|v\|=1}
\la\cH_{0}(0, N)v,v\ra_{0}
\\
&= \left(1 +\Odr ({\kappa} ) \right)
\left(1 +(\frac{2\pi}{L})^2 \right)
= 1 +\left(\frac{2\pi}{L}\right)^2     +\Odr ({\kappa} )
\end{align*}
since $L\geqslant 1$.
\end{proof}

The proven lemma implies that there exists $C>0$ such that the
set $\Xi:=\{\l\in\CC: |\l-1|<C{\kappa}\}$ contains no
eigenvalues of $\cH_{{\kappa}}(0, N)$ except $\lambda^{\per}({\kappa})$.
Now for $\lambda\in \Xi$:
\begin{align*}
|\widehat{\l}^{p}({\kappa})-\lambda|
&\geqslant
| 1+4\pi^2 L^{-2}-1 |  +|1+4\pi^2 L^{-2}- \widehat{\l}^{p}({\kappa})|  +|\lambda-1|
\\
&\geqslant
4\pi^2 L^{-2}-2C {\kappa}
\end{align*}
Thus  the distance from this set $\Xi$ to the remaining part of the spectrum
$\sigma(\cH_{{\kappa}}(0, N))\setminus \{ {\l}^{p}({\kappa})\}$ is estimated from below by $
4\pi^2 L^{-2}-2C {\kappa}$.
Theorem \ref{th4.2} concerns values of ${\kappa}$ which are much smaller than the inverse length $1/L$ of the waveguide segment.
Thus  we can certainly assume ${\kappa} \leqslant \pi^2/(C \,L^2)$ which means that $ 4\pi^2 L^{-2}-2C {\kappa}\geqslant 3\pi^2 L^{-2}$ is positive.
By \cite[Ch. V, Sec. 3.5]{Kato-66} it implies that for $\l\in\Xi$ the identity
\begin{equation}\label{4.21}
(\cH_{{\kappa}}(0, N)-\l)^{-1}=
\frac{\la\cdot,\widetilde{\psi}^{\per}\ra_{{\kappa}}}{\lambda^{\per}({\kappa})-\l}
\widetilde{\psi}^{\per}+\mathcal{R}_{{\kappa}}(\lambda)
\end{equation}
holds true, where $\widetilde{\psi}^{\per}:=N^{-1/2}\psi^{\per}$,
$\la\cdot,\cdot\ra_{{\kappa}}:=\la\cdot,\cdot\ra_{L_2(D_{{\kappa}}(N))}$,
$\mathcal{R}_{{\kappa}}$ is holomorphic w.r.t. to $\l\in\Xi$
as an operator from $L_2(D_\rho(N))$ to ${W_2}^2(D_\rho(N))$.
The factor $N^{-1/2}$ in the definition of $\widetilde{\psi}^{\per}$
provides the normalization
$\|\widetilde{\psi}^{\per}\|_{L_2(D_{0}(N))}=1$.
For $f\in L_2(D_{{\kappa}}(N))$ the function $u=\mathcal{R}_{{\kappa}}(\lambda)f$
solves the equation
\begin{align}
&\big(\cH_{{\kappa}}(0, N)-\lambda^{\per}({\kappa})\big)u=\widehat{f}_{{\kappa}},
\quad \widehat{f}_{{\kappa}}:=f-\widetilde{\psi}^{\per}\la f,
\widetilde{\psi}^{\per}\ra_{L_2(D_{{\kappa}}(N))},\label{4.38}
\\
&\|f_{{\kappa}}\|_{L_2(D_{{\kappa}}(N))}\leqslant \|f\|_{L_2(D_{{\kappa}}(N))}.\label{4.7}
\end{align}
For each $f\in L_2(D_{{\kappa}}(N))$ the function $\mathcal{R}_{{\kappa}}(\l)f$
is orthogonal to $\widetilde{\psi}_0^{\per}$ in $L_2(D_{{\kappa}}(N))$.

The formulas (\ref{4.2}), (\ref{4.21}) imply
\begin{align}
&(\cT({\kappa})\cH_{{\kappa}}(0, N)\cT^{-1}({\kappa})-\l)^{-1}=
\frac{\la\cdot,P_3\cT({\kappa})\widetilde{\psi}^{\per}\ra_0}{\lambda^{\per}({\kappa})-\l}
\cT({\kappa})\widetilde{\psi}^{\per}+\widehat{\mathcal{R}}_{{\kappa}}(\lambda),
\label{4.62}
\\
&\widehat{\mathcal{R}}_{{\kappa}}(\l):=\cT({\kappa}) \mathcal{R}_{{\kappa}}(\lambda)\cT^{-1}({\kappa}).\nonumber
\end{align}
Let $\cH_{0}(0, N)$ be the Laplacian on $D_0(N)$ subject to Dirichlet boundary condition on $\G_0(N)$
and to Neumann one on $\g_0(N)$.

\medskip

For the subsequent estimates in the perturbation theory we will need some a-priori bounds on several `tame' operators, like resolvents.
These are gathered in the next

\begin{lemma}\label{lm4.3}
For all $\l\in\Xi$ and all $f\in L_2(D_{{\kappa}}(N))$ the estimates
\begin{align}
&\|\mathcal{R}_{{\kappa}}(\lambda)f\|_{{W_2}^2(D_{\rho}(N))}  \leqslant
CL^2\|f\|_{L_2(D_{{\kappa}}(N))}\label{4.40},
\\
&\|\widehat{\mathcal{R}}_{{\kappa}}(\lambda)f
-\mathcal{R}_0(1)\cT ({\kappa})f\|_{{W_2}^2(D_{0}(N))}  \leqslant
C{\kappa}L^4\|f\|_{L_2(D_{{\kappa}}(N))}\label{4.41}
\end{align}
hold true, where the constants $C$ are independent of $L$, $l$, $N$, $a$, $g$, and $f$.
\end{lemma}

\begin{proof}
Given and $f\in L_2(D_{{\kappa}}(N))$, we introduce $f_{{\kappa}}$ and $u$ by (\ref{4.38}). It follows from the definition of
$\Xi$ and $\mathcal{R}_{{\kappa}}(\l)$ and (\ref{4.7}) that
\begin{equation}\label{4.6}
\|u\|_{L_2(D_{{\kappa}}(N))}\leqslant CL^2\|f\|_{L_2(D_{{\kappa}}(N))}.
\end{equation}
Let $\chi=\chi(t)$ be an infinitely differentiable cut-off
function equalling one for $t<1/4$ and zero for $t>1/2$. We denote
\begin{equation*}
\phi(\xi,{\kappa}):=\chi(\xi_1)\E^{\xi_1 h(\xi_2,{\kappa})}+
\chi(L-\xi_1)\E^{(L-\xi_1)h(\xi_2,{\kappa})}+2-\chi(\xi_1)-
\chi(L-\xi_1).
\end{equation*}
By (\ref{4.36}) we have
\begin{equation}\label{4.13}
\|\phi-1\|_{C^2(\overline{D_0(N)})}  \leqslant   C{\kappa}.
\end{equation}
We construct $u$ as $u=\cT^{-1}({\kappa})\phi \widetilde{u}$.
We substitute this identity and (\ref{4.4}) into the equation (\ref{4.38}) which implies that for $\widetilde{u}$
\begin{equation}\label{4.14}
(\cH_{0}(0, N)+\mathcal{L}_{{\kappa}})\widetilde{u}=
\widetilde{u}+\phi^{-1} \cT({\kappa}) f_{{\kappa}},
\end{equation}
while $\mathcal{L}_{{\kappa}}$ is a second order differential
operator such that
\begin{equation}\label{4.10}
\|\mathcal{L}_{{\kappa}} v\|_{{W_2}^2(D_{0}(N))}  \leqslant   C{\kappa}
\|v\|_{{W_2}^2(D_{0}(N))}.
\end{equation}
It follows from (\ref{4.6}) that
\begin{equation}\label{4.11}
\|\widetilde{u}\|_{L_2(D_0(N))}\leqslant C L^2\|f\|_{L_2(D_0(N))}.
\end{equation}
Reproducing word by word the proof of Lemma~7.1 in \cite[Ch. I\!I\!I, Sec. 7]{LadyzhenskayaU-73} and employing (\ref{4.11}), one can prove easily an estimate
\begin{equation}\label{4.12}
\|v\|_{{W_2}^2(D_0(N))}\leqslant C\|\cH_{0}(0, N) v\|_{L_2(D_0(N))}.
\end{equation}
Combining this estimate with (\ref{4.10}), we can solve the equation for $\widetilde{u}$ as follows,
\begin{equation*}
\widetilde{u}=\cH_{0}(0, N)^{-1}(\I-\mathcal{L}_{{\kappa}} \cH_{0}(0, N)^{-1})^{-1}(
\widetilde{u}+\phi^{-1}\cT({\kappa}) f_{{\kappa}}),
\end{equation*}
where all the operators are well-defined. Applying the estimates (\ref{4.11}), (\ref{4.12}) once again, we arrive at the inequality
\begin{equation*}
\|\widetilde{u}\|_{{W_2}^2(D_0(N))}\leqslant C L^2\|f\|_{L_2(D_0(N))},
\end{equation*}
which implies (\ref{4.40}).

We proceed to the proof of (\ref{4.41}). We rewrite (\ref{4.14}) as
\begin{equation}\label{4.15}
(\cH_{0}(0, N)-1)\widetilde{u}=
\phi^{-1} \cT({\kappa}) f_{{\kappa}}+\mathcal{L}_{{\kappa}}\widetilde{u}.
\end{equation}
In accordance with (\ref{4.38}) the function $u_0:=\mathcal{R}_0(1) \cT({\kappa}) f$ is the solution to the equation
\begin{align}
&(\cH_{0}(0, N)-1)u_0=f_0,\label{4.17}
\\
&f_0:=\cT({\kappa})f - (\cT({\kappa})f, \widetilde{\psi}_0^{\per})_{L_2(D_0(N))}\widetilde{\psi}_0^{\per},
\quad \widetilde{\psi}_0^{\per}:=N^{-1/2}\psi_0.\nonumber
\end{align}
By the definition of $\mathcal{R}_{{\kappa}}(\l)$ and $\mathcal{R}_0(1)$, the function $u=\cT^{-1}({\kappa})\phi\widetilde{u}$
is orthogonal to $\widetilde{\psi}^{\per}$ in $L_2(D_{{\kappa}}(N))$, while $u_0$ is orthogonal to $\widetilde{\psi}_0^{\per}$ in $L_2(D_0(N))$.
Hence, by the definition of $\phi$ and Lemma~\ref{lm4.1},
\begin{equation}\label{4.16}
\begin{gathered}
\widetilde{u}=\widehat{u}+ \widetilde{\psi}_0^{\per}(\widetilde{u}, \widetilde{\psi}_0^{\per})_{L_2(D_0(N))},
\\
(\widehat{u}, \widetilde{\psi}_0^{\per})_{L_2(D_0(N))}=0,\quad |(\widetilde{u}, \widetilde{\psi}_0^{\per})_{L_2(D_0(N))}|\leqslant C{\kappa}.
\end{gathered}
\end{equation}
The equations (\ref{4.15}), (\ref{4.17}) imply
\begin{equation*}
(\cH_{0}(0, N)-1)(\widehat{u}-u_0)=\phi^{-1} \cT({\kappa})f_{{\kappa}}- \mathcal{L}_{{\kappa}}\widetilde{u}-f_0,
\end{equation*}
where $(\widehat{u}-u_0)$ is orthogonal to $\widetilde{\psi}_0^{\per}$. Since the function $\widehat{u}-u_0$ is well-defined, the right hand side of the last equation is orthogonal to $\widetilde{\psi}_0^{\per}$ in $L_2(D_0(N))$. Moreover, the definition of $\phi$, $f_{{\kappa}}$, $f_0$, Lemma~\ref{lm4.1}, and (\ref{4.10}) yields
\begin{equation*}
\|\phi^{-1} \cT({\kappa})f_{{\kappa}}- \mathcal{L}_{{\kappa}}\widetilde{u}-f_0\|_{L_2(D_0(N))}\leqslant C{\kappa} L^2\|f\|_{L_2(D_{{\kappa}}(N))}.
\end{equation*}
Employing the aforementioned facts, by analogy with (\ref{4.6}), (\ref{4.12}) we get
\begin{equation*}
\|\widehat{u}-u_0\|_{{W_2}^2(D_0(N))}\leqslant C{\kappa} L^2\|f\|_{L_2(D_{{\kappa}}(N))}.
\end{equation*}
Since $\cT({\kappa})\mathcal{R}_{{\kappa}}(\l)f=\cT({\kappa}) u=\widetilde{u}\phi$, the last inequality and (\ref{4.16}) imply (\ref{4.41}).
\end{proof}
\medskip

\section{Deterministic lower bounds: Proof of Theorem \ref{th4.2}}
\label{s:deterministic-proof}

After these preparatory lemmata we turn to the heart of the proof of of Theorem~\ref{th4.2}.
It consists in deteremining the sign and estimating
the three terms of the perturbation $\mathcal{Q}_\rho=\cQ_{11}\frac{\p^2}{\p \xi_1^2}+\cQ_{1}\frac{\p}{\p\xi_1}+\cQ_{2}\frac{\p}{\p\xi_2}$.
In fact, since we have to perform a perturbative analysis not around
the Laplacian $\cH_{0}(0, N)$ of the straight waveguide segment
but rather around the one $\cH_{{\kappa}}(0, N)$
with the maximal possible curvature, the effective perturbation will be
$ \cH_{{\kappa}}(\epsilon, N)-\cH_{{\kappa}}(0, N)= \mathcal{Q}_\rho-\mathcal{Q}_{{\kappa}}$.

As before we will work with the change of the variables $x\mapsto \xi$
introduced above to transform the operator $\cH_{{\kappa}}(\eps, N)$,
\begin{align}
&\cT(\rho)\cH_{{\kappa}}(\eps, N)\cT(\rho)^{-1} =  \cT({\kappa})\cH_{{\kappa}}(0, N)\cT({\kappa})^{-1} +\mathcal{M}(\eps,{\kappa})
\label{4.20}
\\
&\mathcal{M}(\eps,{\kappa})=M_{11}\frac{\p^2}{\p\xi_1^2} +M_{1}\frac{\p}{\p\xi_1}+M_{2}\frac{\p}{\p\xi_2},\nonumber
\\
&M_{11}:=\cQ_{11}(\xi,\rho)-\cQ_{11}(\xi,{\kappa}),\quad
M_{i}:=\cQ_{i}(\xi,\rho)-\cQ_{i}(\xi,{\kappa}),\quad i=1,2.
\end{align}
where  the functions $\cQ_{11}$, $\cQ_i$ have been introduced in (\ref{2.0}).
Let us calculate explicitly $\cQ_2$ for values $\xi_1\in [0,l]$. First we expand the function $\rho\mapsto \cQ_2(\xi, \rho)$ around $\rho=0$:
\begin{align}\label{eq:expansionQ}
 \cQ_2(\xi,\rho) = g''(\xi_1)\rho \left(1+\xi_2 g''(\xi_1) \rho + \Odr(\rho^2)\right)
                 = g''(\xi_1)\rho+\xi_2 (g''(\xi_1))^2 \rho^2 + \Odr(\rho^3).
\end{align}
Suppressing for the moment the $\xi$-dependence and using here the prime ${}'$ to denote derivatives w.r.t. the variable $\rho$
we have
\begin{align*}
 \cQ_2'(\rho)&=\cQ_2'(0) +\cQ_2''(0)\rho +\Odr(\rho^2)\\
 \cQ_2''(\rho) &=\cQ_2''(0) +\Odr(\rho).
\end{align*}
Thus for some intermediate value $\tilde \rho\in [\rho,{\kappa}]$ we obtain
\begin{align*}
 \cQ_2(\rho)-\cQ_2({\kappa}) &=\cQ_2'({\kappa})(\rho-{\kappa}) +\frac{1}{2}\cQ_2''({\kappa}) (\rho-{\kappa})^2 +\frac{1}{6}\cQ_2'''(\tilde\rho) (\rho-{\kappa})^3
 \\
 &= -\cQ_2'(0) \epsilon -\cQ_2''(0) \epsilon {\kappa} + \frac{1}{2}\cQ_2''(0) \epsilon^2  + \Odr({\kappa}^2\epsilon)
\intertext{and after evaluating (\ref{eq:expansionQ}) and inserting the value at $\rho=0$,}
 &= -\frac{\partial^2 g(\xi_1)}{\partial \xi_1^2} \epsilon -\xi_2 \left(\frac{\partial^2 g(\xi_1)}{\partial \xi_1^2}\right)^2\epsilon ({\kappa}+{\kappa})
 + \Odr({\kappa}^2\epsilon)
\end{align*}
follows, where we have made the dependence on $\xi$ explicit, again.
Direct calculations of similar type show that
\begin{align}
&M_{11}:=M_{11}^{(0)}+M_{11}^{(1)}+\eps_j{\kappa}^2 M_{11}^{(2)},\quad
M_{i}:=M_{i}^{(0)}+M_{i}^{(1)}+\eps_j{\kappa}^2 M_{i}^{(2)},\label{4.43}
\\
&M_{11}^{(0)}=2\eps_j\xi_2 g''_j,\quad
M_{11}^{(1)}=-\eps_j({\kappa}+\rho_j)\big((g'_j)^2-3\xi_2^2(g''_j)^2\big),\nonumber
\\
&M_{1}^{(0)}=\eps_j\xi_2 g'''_j,\quad
M_{1}^{(1)}=-\eps_j({\kappa}+\rho_j)(g'_j-3\xi_2^2g'''_j)g''_j,\nonumber
\\
&M_{2}^{(0)}=-\eps_j g''_j,\quad
M_{2}^{(1)}=-\eps_j({\kappa}+\rho_j)\ \xi_2(g''_j)^2,\nonumber
\end{align}
for $\xi_1\in[(j-1)l,jl]$, and
\begin{equation}\label{4.22}
|M_{11}^{(2)}|+|M_{1}^{(2)}|+|M_{2}^{(2)}|  \leqslant C,
\end{equation}
where the constant $C$ is independent of $L$, $N$, $l$, $j$, $a$, $\xi$, and $g$.

Consider $\l\in\Theta$. Employing the technique used in
\cite[Sect. 4]{Borisov-06}, \cite{Gadylshin-02} and (\ref{4.21}), (\ref{4.20}),
one can show easily that $\l=\lambda(\rho)\in\Theta$ is an eigenvalue of $\cH_{{\kappa}}(\eps, N)$ if and only if it solves the equation
\begin{equation*}
\l-\lambda^{\per}(\rho)
=\la (\I+\mathcal{M}(\eps,{\kappa})\widehat{\mathcal{R}}_{{\kappa}}(\l))^{-1} \mathcal{M}(\eps,{\kappa})\widehat{\psi}^{\per}, P_3 \widehat{\psi}^{\per}\ra,
\end{equation*}
where $\widehat{\psi}^{\per}:=\cT({\kappa})\widetilde{\psi}^{\per}$,
$\la\cdot,\cdot\ra:=\la\cdot,\cdot\ra_0$, $P_3$ is taken for $\rho=({\kappa},\ldots,{\kappa})$.
 We rewrite this equation as
\begin{equation}
\begin{aligned}
\lambda(\rho)-\lambda^{\per}({\kappa})=&\la  \mathcal{M}(\eps,{\kappa})\widehat{\psi}^{\per}, P_3\widehat{\psi}^{\per}\ra  - \la \mathcal{M}(\eps,{\kappa}) \widehat{\mathcal{R}}_{{\kappa}}(\lambda(\rho)) \mathcal{M}(\eps,{\kappa})\widehat{\psi}^{\per}, P_3\widehat{\psi}^{\per}
\ra
\\
&+ \la  \big(\mathcal{M}(\eps,{\kappa}) \widehat{\mathcal{R}}_{{\kappa}}(\lambda(\rho))\big)^2 (\I+\mathcal{M}(\eps,{\kappa})\widehat{\mathcal{R}}_{{\kappa}}(\l))^{-1}
\mathcal{M}(\eps,{\kappa})\widehat{\psi}^{\per}, P_3\widehat{\psi}^{\per}
\ra.
\end{aligned}\label{4.23}
\end{equation}
The reduced resolvent $\widehat{\mathcal{R}}_{{\kappa}}$ satisfies the resolvent identity
\begin{equation*}
\widehat{\mathcal{R}}_{{\kappa}}(\l)-\widehat{\mathcal{R}}_{{\kappa}}(\mu)=(\l-\mu) \widehat{\mathcal{R}}_{{\kappa}}(\mu) \widehat{\mathcal{R}}_{{\kappa}}(\l).
\end{equation*}
Substituting this identity with $\l=\lambda(\rho)$, $\mu=1$ and (\ref{4.43}) into (\ref{4.23}),  and using the fact that $\widehat{\mathcal{R}}_{{\kappa}}(1)=\mathcal{R}_0(1)$, we obtain
\begin{align}
&\lambda(\rho)-\lambda^{\per}({\kappa})=S_1(\eps,{\kappa})+(\lambda(\rho)-\lambda^{\per}({\kappa})) S_2(\eps,{\kappa}) + S_3(\eps,{\kappa}),
\label{4.24}
\\
& S_1(\eps,{\kappa})=\la \widehat{\mathcal{M}}(\eps,{\kappa})\widehat{\psi}^{\per}, \widehat{P}_3\widehat{\psi}^{\per}  \ra  -\la
\widehat{\mathcal{M}}(\eps,{\kappa})\mathcal{R}_0(1)
\widehat{\mathcal{M}}(\eps,{\kappa})\widehat{\psi}^{\per},\widehat{\psi}^{\per}
\ra,\nonumber
\\
&S_2(\eps,{\kappa})=-\la
 \mathcal{M} (\eps,{\kappa}) \widehat{\mathcal{R}}_{{\kappa}}(\lambda^{\per}({\kappa})) \widehat{\mathcal{R}}_{{\kappa}}(\lambda(\rho))
 \mathcal{M}(\eps,{\kappa}) \widehat{\psi}^{\per},P_3\widehat{\psi}^{\per}
\ra \nonumber
\\
&S_3(\eps,{\kappa})=
- \la \mathcal{M}^{(2)}(\eps,{\kappa})  \big(\I- \widehat{\mathcal{R}}_{{\kappa}}(\lambda^{\per}({\kappa})) \mathcal{M}(\eps,{\kappa}) \big) \widehat{\psi}^{\per},P_3\widehat{\psi}^{\per}
\ra \nonumber
\\
&\hphantom{S_3(\eps,{\kappa})=} -\la \widehat{\mathcal{M}}(\eps,{\kappa}) \widehat{\mathcal{R}}_{{\kappa}}(\lambda^{\per}({\kappa})) \mathcal{M}^{(2)}(\eps,{\kappa})  \widehat{\psi}^{\per},P_3\psi^{\per}
\ra\nonumber
\\
&\hphantom{S_3(\eps,{\kappa})=}
+ \la  \big(\mathcal{M}(\eps,{\kappa}) \widehat{\mathcal{R}}_{{\kappa}}(\lambda(\rho))\big)^2 (\I+\mathcal{M}(\eps,{\kappa})\widehat{\mathcal{R}}_{{\kappa}}(\l))^{-1}
\mathcal{M}(\eps,{\kappa})\widehat{\psi}^{\per},P_3\widehat{\psi}^{\per}
\ra,\nonumber
\\
&\hphantom{S_3(\eps,{\kappa})=} -\la \widehat{\mathcal{M}}(\eps,{\kappa}) \big(\widehat{\mathcal{R}}_{{\kappa}}(\lambda(\rho))-\mathcal{R}_0(1)\big) \widehat{\mathcal{M}}(\eps,{\kappa})\widehat{\psi}^{\per},P_3\widehat{\psi}^{\per}
\ra\nonumber
\\
&\hphantom{S_3(\eps,{\kappa})=}+\la \widehat{\mathcal{M}}(\eps,{\kappa}) \widehat{\psi}^{\per},(P_3-\widehat{P}_3)\psi^{\per}
\ra\nonumber
\\
&\hphantom{S_3(\eps,{\kappa})=}-\la
\widehat{\mathcal{M}}(\eps,{\kappa})\mathcal{R}_0(1)
\widehat{\mathcal{M}}(\eps,{\kappa})\widehat{\psi}^{\per},(P_3-1)\widehat{\psi}^{\per}
\ra,\nonumber
\\
&\widehat{P}_3(\xi):=1-\xi_2 G''(\xi_1,{\kappa},\ldots,{\kappa}),\nonumber
\\
&\widehat{\mathcal{M}}(\eps,{\kappa}):=\mathcal{M}^{(0)}(\eps,{\kappa}) +\mathcal{M}^{(1)}(\eps,{\kappa}), \nonumber
\\
&\mathcal{M}^{(i)}(\eps,{\kappa}):=M_{11}^{(i)}\frac{\p^2}{\p\xi_1^2} +M_{1}^{(i)}\frac{\p}{\p\xi_1}
+M_{2}^{(i)}\frac{\p}{\p\xi_2},\quad i=0,1,2. \nonumber
\end{align}
By (\ref{4.22}), (\ref{2.1}) and Lemmas~\ref{lm4.1},~\ref{lm4.3} we get the estimate for $S_3(\eps,{\kappa})$:
\begin{equation}\label{4.25}
|S_3(\eps,{\kappa})|\leqslant C{\kappa}^2 |\eps|_2 L^{9/2},\quad
\text{ where } \quad |\eps|_2:=\left(\sum\limits_{i=1}^{N}\eps_i^2\right)^{1/2}.
\end{equation}
It follows from (\ref{2.3}), (\ref{2.1}), the definition of $\mathcal{M}(\eps,{\kappa})$ and Lemma~\ref{lm4.3} that
\begin{equation*}
|S_2|\leqslant C{\kappa}^2 L^2\leqslant C c L^{-7}.
\end{equation*}
Hence, for $\delta$ small enough we have $|S_2|<1/2$, and by (\ref{4.24}), (\ref{4.23}) we get
\begin{equation}\label{4.42}
\lambda(\rho)-\lambda^{\per}({\kappa})
\geqslant\frac{S_1(\eps,{\kappa})}{2}-C{\kappa}^2|\eps|_2 L^{9/2}.
\end{equation}
In order to control the contribution $S_1$ we split it into two parts,
using  thereby identity (\ref{4.43}) and Lemmas~\ref{lm4.1}
\begin{align}
S_1(\eps,{\kappa})=&S_4(\eps,{\kappa})+S_5(\eps,{\kappa}),\label{4.44}
\\
S_4(\eps,{\kappa}):=&N^{-1}\la
\widehat{\mathcal{M}}(\eps,{\kappa})\psi_0^{\per},\psi_0^{\per}\ra +
{\kappa}N^{-1} \la \mathcal{M}^{(0)} \psi_1^{\per},\psi_0^{\per} \ra\nonumber
\\
&+
{\kappa}N^{-1} \la \mathcal{M}^{(0)} \psi_0^{\per},\psi_1^{\per} \ra-N^{-1}\la \mathcal{M}^{(0)}(\eps,{\kappa}) \psi_0^{\per}, \xi_2 G'' \psi_0^{\per}\ra \nonumber
\\
&-N^{-1}\la\mathcal{M}^{(0)}(\eps,{\kappa}) \mathcal{R}_0(1) \mathcal{M}^{(0)}(\eps,{\kappa}) \psi_0^{\per},\psi_0^{\per} \ra,\nonumber
\\
S_5(\eps,{\kappa}):=&S_1(\eps,{\kappa})-S_4(\eps,{\kappa}),\nonumber
\end{align}
where $G''$ is taken for $\rho=({\kappa},\ldots,{\kappa})$.
By Lemma~\ref{lm4.3} and (\ref{4.22}) we can estimate $S_5$,
\begin{equation}\label{4.32}
 |S_5(\eps,{\kappa})|  \leqslant   C{\kappa}^2|\eps|_2L^{7/2}.
\end{equation}
We shall show that $S_5$ can be regarded as a small order perturbation of the main term $S_4$. In order to it, we need to calculate $S_4$. It is easy to check that
\begin{equation} \label{4.45}
\begin{aligned}
N^{-1}\la \widehat{\mathcal{M}} (\eps,{\kappa})\psi_0^{\per},\psi_0^{\per}\ra
&=
\frac{\|g''\|_{L_2(0,l)}^2}{2L}\sum\limits_{j=1}^{N}\eps_j({\kappa}+\rho_j),
\\
-N^{-1}\la \mathcal{M}^{(0)}(\eps,{\kappa}) \psi_0^{\per}, \xi_2 G'' \psi_0^{\per}\ra
&=-\frac{\|g''\|_{L_2(0,l)}^2}{2L}{\kappa}\sum\limits_{j=1}^{N}\eps_j.
\end{aligned}
\end{equation}
By (\ref{4.43}) we obtain
\begin{align*}
\la
\mathcal{M}^{(0)}(\eps,{\kappa})&\psi_1^{\per},\psi_0^{\per}\ra+\la
\mathcal{M}^{(0)}(\eps,{\kappa})\psi_0^{\per},\psi_1^{\per}\ra
\\
=&\sum\limits_{j=1}^N
\eps_j \left\la 2\xi_2g''_j\frac{\p^2\psi_1^{\per}}{\p\xi_1^2}+
\xi_2 g'''_j \frac{\p\psi_1^{\per}}{\p\xi_1}
- g''_j \frac{\p\psi_1^{\per}}{\p\xi_2},\psi_0^{\per}
\right\ra_{L_2(D_{0}(j-1,j))}
\\
&-\sum\limits_{j=1}^N
\eps_j \left\la\xi_2 g''_j \frac{\p\psi_0^{\per}}{\p\xi_2},\psi_1^{\per} \right\ra_{L_2(D_0(j-1,j))}
\\
=&\bigg( \left\la 2\xi_2g''\frac{\p^2\psi_1^{\per}}{\p\xi_1^2}+\xi_2 g''' \frac{\p\psi_1^{\per}}{\p\xi_1} - \xi_2g'' \frac{\p\psi_1^{\per}}{\p\xi_2},\psi_0^{\per}
\right\ra_{L_2(D_0(1))}
\\
&- \left\la \psi_1^{\per}, \xi_2 g'' \frac{\p\psi_0^{\per}}{\p\xi_2}  \right\ra_{L_2(D_0(1))}
\bigg) \sum\limits_{j=1}^{N}\eps_j.
\end{align*}
We take into the account the equation (\ref{4.34}) for $\psi_1^{\per}$ and integrate by parts,
\begin{equation}\label{4.46}
\begin{aligned}
&\Big\la 2\xi_2 g''\frac{\p^2 \psi_1^{\per}}{\p\xi_1^2}+\xi_2
g'''\frac{\p \psi_1^{\per}}{\p\xi_1}- \xi_2 g''\frac{\p
\psi_1^{\per}}{\p\xi_2},\psi_0^{\per} \Big\ra_{L_2(D_{0}(1))} -
\Big\la \psi_1^{\per}, \xi_2 g''\frac{\p\psi_0^{\per}}{\p\xi_2} \Big\ra_{L_2(D_{0}(1))}
\\
&
=\Big\la \xi_2 g''\frac{\p^2 \psi_1^{\per}}{\p\xi_1^2}, \psi_0^{\per} \Big\ra_{L_2(D_{0}(1))}-
\Big\la \xi_2 g'',\frac{\p}{\p\xi_2} \psi_0^{\per}\psi_1^{\per} \Big\ra_{L_2(D_{0}(1))}
\\
&
=
-\Big\la \xi_2 g''\Big(
\frac{\p^2 \psi_1^{\per}}{\p\xi_2^2}+\psi_1^{\per}  -g''\frac{\p\psi_0^{\per}}{\p\xi_2}\Big), \psi_0^{\per} \Big\ra_{L_2(D_{0}(1))}
+ \la g'' \psi_1^{\per},\psi_0^{\per}\ra_{L_2(D_{0}(1))}
\\
&=-\frac{\|g''\|_{L_2(0,l)}^2}{2l} - 2\Big\la g''\psi_1^{\per}, \frac{\p\psi_0^{\per}}{\p\xi_2}  \Big\ra_{L_2(D_{0}(1))}
+ \la g'' \psi_1^{\per},\psi_0^{\per}\ra_{L_2(D_{0}(1))}.
\end{aligned}
\end{equation}

Now consider the last contribution to $S_4$.
Denote $u:=\mathcal{R}_0(1)\mathcal{M}^{(0)}(\eps,{\kappa})\psi_0$. This
function solves the equation
\begin{equation}\label{4.48}
(\cH_{0}(0, N)-1)u=-\sum\limits_{j=1}^N\eps_j
g''_j\frac{\p\psi_0^{\per}}{\p\xi_2}
\end{equation}
and is orthogonal to $\psi_0^{\per}$ in $L_2(D_0(L))$.
We use this equation and proceed as in (\ref{4.46}), to calculate
\begin{align*}
-\la
\mathcal{M}^{(0)}(\eps,{\kappa})&\mathcal{R}_0(1)\mathcal{M}^{(0)}
(\eps,{\kappa})\psi_0^{\per},\psi_0^{\per}\ra
\\
&=-\sum\limits_{j=1}^{N}\eps_j
\Big\la 2\xi_2 g''_j\frac{\p^2 u}{\p\xi_1^2}+\xi_2
g'''_j\frac{\p u}{\p\xi_1}- \xi_2 g''_j\frac{\p
u}{\p\xi_2},\psi_0^{\per} \Big\ra_{L_2(D_{0}(j-1,j))}
\\
&=-\sum\limits_{j=1}^{N}\eps_j
\Big\la \xi_2 g''_j\frac{\p^2 u}{\p\xi_1^2} - \xi_2 g''_j \frac{\p
u}{\p\xi_2},\psi_0^{\per} \Big\ra_{L_2(D_{0}(j-1,j))}
\\
&=\sum\limits_{j=1}^{N}\eps_j
\Big\la \xi_2 g''_j \Big(-\eps_j g''_j \frac{\p\psi_0^{\per}}{\p\xi_2} + \frac{\p^2 u}{\p\xi_2^2}+\frac{\p
u}{\p\xi_2} + u\Big)
,\psi_0^{\per} \Big\ra_{L_2(D_{0}(j-1,j))}
\\
&=\frac{\|g''\|_{L_2(0,l)}^2}{2l} \sum\limits_{j=1}^{N}\eps_j^2 +2\sum\limits_{j=1}^{N}\eps_j \Big\la g''_j u, \frac{\p\psi_0^{\per}}{\p\xi_2}\Big\ra_{L_2(D_0(j-1,j))}
\\
&\hphantom{=}
- \sum\limits_{j=1}^{N}\eps_j \Big\la g''_j u, \frac{\p}{\p \xi_2} (\xi_2 \psi_0^{\per}) \Big\ra_{L_2(D_0(j-1,j))}.
\end{align*}
We substitute the last identity and (\ref{4.45}),  (\ref{4.46}) into (\ref{4.44}),
\begin{equation}\label{4.50}
\begin{aligned}
S_4(\eps,{\kappa})=&
2 N^{-1} \sum\limits_{j=1}^{N} \eps_j \Big\la g''_j u,\frac{\p\psi_0^{\per}}{\p\xi_2}\Big\ra_{L_2(D_0(j-1,j))}
\\
&-N^{-1}{\kappa}
\left( 2\Big\la g''\psi_1^{\per},\frac{\p\psi_0^{\per}}{\p\xi_2}
\Big\ra_{L_2(D_0(1))}- \Big\la g'' \psi_1^{\per},\psi_0^{\per}\Big\ra_{L_2(D_0(1))}\right) \sum\limits_{j=1}^{N}\eps_j
\\
&-N^{-1} \sum\limits_{j=1}^{N} \eps_j \Big\la g''_j u, \frac{\p}{\p\xi_2} (\xi_2\psi_0^{\per}) \Big\ra_{L_2(D_0(j-1,j))}.
\end{aligned}
\end{equation}
Note that terms involving $\|g''\|_{L_2(0,l)}$ cancel.
Next we study each term in the obtained expression for $S_4$.

\begin{lemma}\label{lm4.6}
The inequality
\begin{equation}\label{4.30}
S_4(\eps,{\kappa})\geqslant \frac{\|g'\|_{L_2(0,l)}^2}{L} \left(\frac{\pi}{2}-\frac{16a}{\pi^2}\right) {\kappa} \sum\limits_{j=1}^{N}\eps_j
\end{equation}
holds true.
\end{lemma}
\begin{proof}
Since
\begin{equation}\label{4.65}
\frac{\p\psi_0}{\p\xi_2}(\xi_1,\pi-\xi_2)= -\frac{\p\psi_0}{\p\xi_2}(\xi_1,\xi_2),
\end{equation}
the solutions $\psi_1^{\per}$, $u$ to the equations (\ref{4.34}), (\ref{4.48}) satisfy the same identity,
\begin{equation}\label{4.66}
\psi_1^{\per}(\xi_1,\pi-\xi_2)= -\psi_1^{\per}(\xi_1,\xi_2),\quad
u(\xi_1,\pi-\xi_2)= -u(\xi_1,\xi_2).
\end{equation}
It implies
\begin{align}
& \sum\limits_{j=1}^{N} \eps_j \Big\la g''_j u, \frac{\p}{\p\xi_2} (\xi_2\psi_0^{\per}) \Big\ra_{L_2(D_0(j-1,j))}= \frac{\pi}{2}
\sum\limits_{j=1}^{N} \eps_j \Big\la g''_j u, \frac{\p}{\p\psi_0^{\per}} \Big\ra_{L_2(D_0(j-1,j))},\nonumber
\\
&
\begin{aligned}
S_4(\eps,{\kappa})=&
\left(2-\frac{\pi}{2}\right)N^{-1} \sum\limits_{j=1}^{N} \eps_j \Big\la g''_j u,\frac{\p\psi_0^{\per}}{\p\xi_2}\Big\ra_{L_2(D_0(j-1,j))}
\\
&-2N^{-1}{\kappa}
 \Big\la\psi_1^{\per},g''\frac{\p\psi_0^{\per}}{\p\xi_2}
\Big\ra_{L_2(D_0(1))} \sum\limits_{j=1}^{N}\eps_j.
\end{aligned}\label{4.81}
\end{align}
Let us estimate the first term in (\ref{4.66}). We integrate by parts taking into consideration (\ref{4.65}), (\ref{4.66}),
\begin{equation}
\begin{aligned}
&\sum\limits_{j=1}^{N} \eps_j \Big\la u, g''_j \frac{\p\psi_0^{\per}}{\p\xi_2}\Big\ra_{L_2(D_0(j-1,j))}=
\sum\limits_{j=1}^{N} \eps_j \Big\la \frac{\p^2 u}{\p\xi_1^2},g_j\frac{\p\psi_0^{\per}}{\p\xi_2}\Big\ra_{L_2(D_0(j-1,j))}
\\
& =
\sum\limits_{j=1}^{N} \eps_j^2 \Big\la g_j''\frac{\p\psi_0 u}{\p\xi_2},g_j\frac{\p\psi_0^{\per}}{\p\xi_2}\Big\ra_{L_2(D_0(j-1,j))}
- \sum\limits_{j=1}^{N} \eps_j^2 \Big\la \frac{\p^2 u}{\p\xi_2}+u,g_j\frac{\p\psi_0^{\per}}{\p\xi_2}\Big\ra_{L_2(D_0(j-1,j))}
\\
& =
-\frac{\|g'\|_{L_2(0,l)}^2}{l}\sum\limits_{j=1}^{N} \eps_j^2 +2
\Big(\frac{\p u}{\p \xi_2}(\cdot,\pi),\sum\limits_{j=1}^{N} \eps_j g_j\Big)_{L_2(0,L)}.
\end{aligned}\label{4.67}
\end{equation}
We construct the function $u$ by the separation of variables,
\begin{align}
&\frac{\p\psi_0^{\per}}{\p\xi_2}
=\sum\limits_{m=1}^{\infty} \a_m\sin 2m \xi_2,\quad \a_m=\sqrt{\frac{2}{\pi l}}\frac{4m}{4m^2-1},\nonumber,
\\
&u(\xi)=\sum\limits_{m=1}^{\infty} \a_m U_m(\xi_2)\sin 2m\xi_1,\label{4.68}
\end{align}
where the functions $U_m$ solve the boundary value problems,
\begin{equation*}
-U''_m+(4m^2-1)U_m=-\sum\limits_{j=1}^{N}\eps_j g''_j\quad\text{in}\quad L_2(0,L),\qquad U'_m(0)=U'_m(L)=0.
\end{equation*}
It follows from (\ref{4.67}), (\ref{4.68}) that
\begin{equation}\label{4.69}
\begin{aligned}
\sum\limits_{j=1}^{N}\eps_j \Big\la u,g''_j\frac{\p\psi_0^{\per}}{\p\xi_2}\Big\ra=-\frac{\|g'\|_{L_2(0,l)}^2}{l} \sum\limits_{j=1}^{N}\eps_j^2 + 4 \sum\limits_{m=1}^{\infty} m\a_m \Big( \sum\limits_{j=1}^{N}\e_j g_j, U_m\Big)_{L_2(0,L)}. \end{aligned}
\end{equation}
We represent $U_m$ as
\begin{equation*}
U_m=\sum\limits_{j=1}^{N}\eps_j g_j +(4m^2-1) \widetilde{U}_m,
\end{equation*}
where $\widetilde{U}_m$ is the solution to the problem
\begin{equation*}
-\widetilde{U}''_m+(4m^2-1)\widetilde{U}_m=-\sum\limits_{j=1}^{N}\eps_j g_j\quad\text{in}\quad L_2(0,L),\qquad \widetilde{U}'_m(0)=\widetilde{U}'_m(L)=0.
\end{equation*}
Thus,
\begin{align*}
&\|\widetilde{U}'_m\|_{L_2(0,l)}^2+(4m^2-1) \|\widetilde{U}_m\|_{L_2(0,l)}^2=-\Big(\sum\limits_{j=1}^{N}\eps_j g_j, \widetilde{U}_m\Big)_{L_2(0,L)},
\\
&\Big( \sum\limits_{j=1}^{N}\e_j g_j, U_m\Big)_{L_2(0,L)}=\|g\|_{L_2(0,l)}^2 \sum\limits_{j=1}^{N} \eps_j^2 + (4m^2-1) \Big( \sum\limits_{j=1}^{N}\e_j g_j, \widetilde{U}_m\Big)_{L_2(0,L)},
\\
&(4m^2-1)\|\widetilde{U}_m\|_{L_2(0,L)}\leqslant \Big\|
\sum\limits_{j=1}^{N}\eps_j g_j\Big\|_{L_2(0,L)}= \|g\|_{L_2(0,l)} \sum\limits_{j=1}^{N} \eps_j^2.
\end{align*}
Two last relations imply
\begin{equation*}
\Big( \sum\limits_{j=1}^{N}\e_j g_j, U_m\Big)_{L_2(0,L)}\geqslant 0,
\end{equation*}
and by (\ref{4.69}) it yields
\begin{equation}\label{4.70}
\sum\limits_{j=1}^{N}\eps_j \Big\la u, g''_j \frac{\p\psi_0}{\p\xi_2}\Big\ra_{L_2(D_0(j-1,j))} \geqslant -\frac{\|g\|_{L_2(0,l)}^2}{l} \sum\limits_{j=1}^{N}\eps_j^2.
\end{equation}

We proceed to the second term in the right hand side of (\ref{4.66}). We construct $\psi_1^{\per}$ as
\begin{equation}\label{4.71}
\psi_1^{\per}=g\frac{\p\psi_0^{\per}}{\p\xi_2}+\sqrt{\frac{2}{\pi l}}\widetilde{\psi}_1^{\per},
\end{equation}
where $ \widetilde{\psi}_1^{\per}$ solves the boundary value problem
\begin{equation}\label{4.72}
(\D+1)\widetilde{\psi}_1^{\per}=0\quad\text{in}\quad D_0(1),\qquad \widetilde{\psi}_1^{\per}=g\quad\text{as}\quad{\xi_2=\pi},\qquad \widetilde{\psi}_1^{\per}=-g\quad\text{as}\quad{\xi_2=0},
\end{equation}
and satisfies periodic boundary condition on $\xi_1=0$ and $\xi_1=l$. Hence,
\begin{equation}\label{4.73}
\begin{aligned}
-\Big\la\psi_1^{\per},g''\frac{\p\psi_0^{\per}}{\p\xi_2}
\Big\ra_{L_2(D_0(1))}=&\frac{\|g'\|_{L_2(0,l)}^2}{l}- \sqrt{\frac{2}{\pi l}}
\Big\la\widetilde{\psi}_1^{\per},g''\frac{\p\psi_0^{\per}}{\p\xi_2}
\Big\ra_{L_2(D_0(1))}
\\
=&\frac{\|g'\|_{L_2(0,l)}^2}{l}+\sqrt{\frac{2}{\pi l}}
\Big\la \frac{\p\widetilde{\psi}_1^{\per}}{\p\xi_1}, g'\frac{\p\psi_0^{\per}}{\p\xi_2}
\Big\ra_{L_2(D_0(1))}.
\end{aligned}
\end{equation}
Now we construct $\widetilde{\psi}_1^{\per}$ by the separation of variables,
\begin{equation}\label{4.74}
\frac{\p\widetilde{\psi}_1^{\per}}{\p\xi_1}=\sum\limits_{m=1}^{\infty} \left(
\b_m^{(+)}\cos\frac{2\pi m}{l}\xi_1+ \b_m^{(-)} \sin\frac{2\pi m}{l}\xi_1\right) V_m(\xi_2),
\end{equation}
where $\b_m^{(\pm)}$ are introduced as the coefficients of the Fourier series for $g'$,
\begin{equation}\label{4.77}
g'(\xi_1)=\sum\limits_{m=1}^{\infty} \left(
\b_m^{(+)}\cos\frac{2\pi m}{l}\xi_1+ \b_m^{(-)} \sin\frac{2\pi m}{l}\xi_1\right).
\end{equation}
The functions $V_m$ are defined as the solutions to the boundary value problems
\begin{equation}\label{4.75}
-V''_m+\left(\frac{4\pi^2 m^2}{l^2}-1\right) V_m=0\quad \text{in}\quad (0,\pi),\qquad V_m(0)=-1,\quad V_m(\pi)=1.
\end{equation}
The functions $V_m$ can be found explicitly,
\begin{equation}\label{4.76}
\begin{aligned}
& V_m(\xi_2)=\frac{\sinh A_m\left(\xi_2-\frac{\pi}{2}\right)}{\sinh \frac{\pi A_m}{2} },\quad A_m:= \sqrt{\frac{4\pi^2 m^2}{l^2}-1},
&& \text{if}\quad \frac{4\pi^2 m^2}{l^2}>1,
\\
& V_m(\xi_2)=\frac{\sin A_m\left(\xi_2-\frac{\pi}{2}\right)}{\sin \frac{\pi A_m}{2}}, \quad A_m:=\sqrt{1-\frac{4\pi^2 m^2}{l^2}}, &&\text{if}\quad \frac{4\pi^2 m^2}{l^2}<1,
\\
& V_m(\xi_2)=\frac{2}{\pi}\xi_2-1,&&\text{if}\quad \frac{4\pi^2 m^2}{l^2}=1.
\end{aligned}
\end{equation}
We substitute (\ref{4.74}) and (\ref{4.77}) into (\ref{4.73}),
\begin{equation}\label{4.78}
-2\Big\la\psi_1^{\per},g''\frac{\p\psi_0^{\per}}{\p\xi_2}
\Big\ra_{L_2(D_0(1))}=\frac{2\|g'\|_{L_2(0,l)}^2}{l}+ \sum\limits_{m=1}^{\infty} \Big(\big(\b_m^{(+)}\big)^2+\big(\b_m^{(-)}\big)^2\Big) B_m,
\end{equation}
where
\begin{equation}\label{4.79}
\begin{aligned}
&B_m:=-\frac{4A_m\coth \frac{\pi A_m}{2}}{\pi (A_m^2+1) },&&\text{if}\quad \frac{4\pi^2 m^2}{l^2}>1,
\\
&B_m:=-\frac{4A_m\cot \frac{\pi A_m}{2}}{\pi (1-A_m^2) },&&\text{if}\quad \frac{4\pi^2 m^2}{l^2}<1,
\\
&B_m:=-\frac{8}{\pi^2},&&\text{if}\quad \frac{4\pi^2 m^2}{l^2}=1.
\end{aligned}
\end{equation}
In view of the estimates
\begin{equation*}
\frac{4z\coth \frac{\pi z}{2}}{\pi\sqrt{z^2+1}}\leqslant \frac{4}{\pi},\quad z\geqslant 0,\qquad \frac{4z\cot \frac{\pi z}{2}}{\pi\sqrt{1-z^2}}\leqslant \frac{8}{\pi^2},\quad z\in[0,1],
\end{equation*}
the coefficients $B_m$ can be estimated as follows,
\begin{equation*}
0>B_m>-\frac{2l}{\pi^2 m},\quad \text{if}\quad 2m>\frac{l}{\pi},\qquad 0>B_m>-\frac{4l}{\pi^3 m}, \quad \text{if}\quad 2m\leqslant \frac{l}{\pi}.
\end{equation*}
We substitute these estimate in (\ref{4.78}),
\begin{align*}
-2\Big\la  \psi_1^{\per}, g''\frac{\p\psi_0^{\per}}{\p\xi_2} \Big\ra_{L_2(D_0(1))}\geqslant & \frac{2\|g'\|_{L_2(0,l)^2}}{l} - \frac{4l}{\pi^3} \sum\limits_{2m\leqslant \frac{l}{\pi}} \frac{\big(\b_m^{(+)}\big)2+\big(\b_m^{(-)}\big)^2}{m}
\\
&-\frac{2l}{\pi^2} \sum\limits_{2m>\frac{l}{\pi}} \frac{\big(\b_m^{(+)}\big)2+\big(\b_m^{(-)}\big)^2}{m}
\\
\geqslant & \frac{2\|g'\|_{L_2(0,l)}^2}{l} - \frac{2l}{\pi^2}
\sum\limits_{m=1}^{\infty} \frac{\big(\b_m^{(+)}\big)^2+\big(\b_m^{(-)}\big)^2}{m}.
\end{align*}
Together with (\ref{4.70}), (\ref{4.81}) and an obvious estimate $\eps_j\leqslant {\kappa}$ it yields
\begin{equation}\label{4.80}
S_4(\eps,{\kappa})\geqslant \left(\frac{\pi}{2} \|g'\|_{L_2(0,l)}^2-\frac{4l^2}{\pi^2}\sum\limits_{m=1}^{\infty} \frac{\big(\b_m^{(+)}\big)^2+\big(\b_m^{(-)}\big)^2}{m}\right)\frac{{\kappa}}{L}
\sum\limits_{j=1}^{N}\eps_j
\end{equation}
It remains to estimate the last sum in this expression.

We introduce the function
\begin{align*}
h(\xi_1):=&\sum\limits_{m=1}^{\infty} \frac{1}{m}
\left(
\b_m^{(+)}\cos\frac{2\pi m}{l}\xi_1+ \b_m^{(-)} \sin\frac{2\pi m}{l}\xi_1\right)+\frac{2\pi \b_0^+}{l},
\\
\b_0^+:=&\frac{1}{l}\int\limits_{0}^{l} g(\xi_1)\di\xi_1,
\end{align*}
and by the definition (\ref{4.77}) of the coefficients $\b_m^{(\pm)}$ we see that
\begin{align}
&\|g\|_{L_2(0,l)}^2=\frac{l}{2}\sum\limits_{m=1}^{\infty} \left(
\frac{l}{2\pi m}\right)^2 \Big(\big(\b_m^{(+)}\big)^2 +\big(\b_m^{(-)}\big)^2\Big)+\big(\b_0^{(+)}\big)^2 l,\nonumber
\\
&\|h\|_{L_2(0,l)}^2=\frac{l}{2} \sum\limits_{m=1}^{\infty} \frac{\big(\b_m^{(+)}\big)2+\big(\b_m^{(-)}\big)^2}{m} +\left(\frac{2\pi\b_0^{(+)}}{l}\right)^2 =\frac{4\pi^2}{l^2}\|g\|_{L_2(0,l)}^2,\label{4.82}
\\
&\frac{4l^2}{\pi^2} \sum\limits_{m=1}^{\infty} \frac{\big(\b_m^{(+)}\big)^2+\big(\b_m^{(-)}\big)^2}{m} = \frac{8l}{\pi^2} (h,g')_{L_2(0,l)}.
\end{align}
Now we use the fact that $\supp g$ lies in a segment of the size $a$ and an obvious estimate
\begin{equation*}
\|g'\|_{L_2(0,l)}\geqslant \frac{\pi}{a} \|g\|_{L_2(0,l)},
\end{equation*}
to obtain
\begin{align*}
-\frac{4l^2}{\pi^2}\sum\limits_{m=1}^{\infty} \frac{\big(\b_m^{(+)}\big)^2+\big(\b_m^{(-)}\big)^2}{m}&\geqslant -\frac{8l}{\pi^2}\|h\|_{L_2(0,l)} \|g'\|_{L_2(0,l)}
\\
&\geqslant -\frac{16}{\pi} \|g\|_{L_2(0,l)} \|g'\|_{L_2(0,l)} \geqslant -\frac{16 a}{\pi^2} \|g'\|_{L_2(0,l)}^2.
\end{align*}
We substitute this inequality into (\ref{4.80}) and complete the proof.
\end{proof}

We substitute the inequalities (\ref{4.30}), (\ref{4.44}), (\ref{4.32}), and a trivial one
\begin{equation*}
|\eps|_2\leqslant \sum\limits_{j=1}^{N}\eps_j
\end{equation*}
into (\ref{4.42}),
\begin{equation*}
\lambda(\rho)-\lambda^{\per}({\kappa})
\geqslant \left(
\frac{\|g'\|_{L_2(0,l)}^2}{2L} \left(\frac{\pi}{2}-\frac{16a}{\pi^2}\right)
-C{\kappa} L^{9/2}\right){\kappa}\sum\limits_{j=1}^{N}\eps_j.
\end{equation*}
Choosing the constant $\delta$ in (\ref{2.1}) small enough, we arrive at the estimate (\ref{2.2}).

\section*{Acknowledments}
This research was partially supported by the Deutsche Forschungsgemeinschaft
through the Emmy-Noether-Programme.
D.I. Borisov was partially supported by Russian Foundation for Basic
Researches and Federal Task Program ``Scientific and Pedagogical
Staff of Innovative Russia for 2009-2013''.

\def\cprime{$'$}\def\polhk#1{\setbox0=\hbox{#1}{\ooalign{\hidewidth
  \lower1.5ex\hbox{`}\hidewidth\crcr\unhbox0}}}

\end{document}